\documentclass[11pt]{amsart}
\usepackage[colorlinks=true, pdfstartview=FitV, linkcolor=blue, citecolor=blue, urlcolor=blue, breaklinks=true]{hyperref}
\usepackage{amsmath,amsfonts,amssymb,amsthm,mathrsfs,amscd,comment,paralist,stmaryrd,etoolbox,mathtools,enumitem}
\usepackage[usenames,dvipsnames]{color}
\usepackage{booktabs}
\usepackage[all]{xy}
\usepackage{tikz}
\usetikzlibrary{decorations.pathmorphing}

\theoremstyle{plain}
\newtheorem{thm}{Theorem}[section]
\newtheorem{prop}[thm]{Proposition}
\newtheorem{lemma}[thm]{Lemma}
\newtheorem{cor}[thm]{Corollary}

\theoremstyle{definition}
\newtheorem{defn}[thm]{Definition}
\newtheorem{rmk}[thm]{Remark}
\newtheorem{ex}[thm]{Example}

\numberwithin{equation}{section}

\newcommand{\g}{\mathfrak{g}}
\newcommand{\h}{\mathfrak{h}}
\newcommand{\bb}{\mathfrak{b}}

\newcommand{\n}{\mathfrak{n}}
\newcommand{\q}{\mathfrak{q}}
\newcommand{\gl}{\mathfrak{gl}}
\newcommand{\m}{\mathfrak{m}}

\newcommand{\Z}{\mathbb{Z}}
\newcommand{\C}{\mathbb{C}}
\newcommand{\N}{\mathbb{N}}

\DeclareMathOperator{\Hom}{Hom}

\DeclareMathOperator{\End}{End}

\DeclareMathOperator{\wt}{wt}

\newtoggle{comments}
\newtoggle{details}
\newtoggle{prelimnote}
\newtoggle{detailsnote}

\toggletrue{detailsnote}   

\iftoggle{comments}{%
  \newcommand{\comments}[1]{
    \begin{center}
      \parbox{6.5 in}{
        \color{red}
          {\footnotesize \textbf{Comments:} #1}
        \color{black}}
    \end{center}}
}{%
  \newcommand{\comments}[1]{}
}

\iftoggle{details}{%
  \newcommand{\details}[1]{
      \ \\
      \color{OliveGreen}
        \begin{footnotesize}
          \textbf{Details:} #1
        \end{footnotesize}
      \color{black}
      \\
  }
}{%
  \newcommand{\details}[1]{}
}

\iftoggle{prelimnote}{%
  
}{%
  
}

\begin{document}
%

\title{Weyl modules for queer Lie superalgebras}

\author{Saudamini Nayak}
\address{ Department of Mathematics, National Institute of Technology Calicut, NIT Campus P.O., Kozhikode-673 601, India}
\email{saudamini@nitc.ac.in}

\thanks{}

\begin{abstract}
 We define global and local Weyl modules for $q \otimes A$, where $q$ is the queer Lie superalgebra and $A$ is an associative commutative unital $\mathbb{C}-$algebra. We prove that global Weyl modules are universal highest weight objects in certain category upto parity reversing functor $\Pi$. Then with the assumption that $A$ is finitely generated and with a special technical condition which simple root system of $q$ satisfy it is shown that the local Weyl modules are finite dimensional. Further they are universal highest map-weight objects in certain category upto $\Pi$. Finally we prove a tensor product property for local Weyl modules.
\end{abstract}


\maketitle
\thispagestyle{empty}

\setcounter{tocdepth}{1}


%
\section{Introduction}

The theory of Lie superalgebras and their representations have a wide range of applications in many areas of physics and mathematics such as describing supersymmetry, in string theory, conformal field theory and number theory to name a few. In 1977, Kac classified the simple Lie superalgebras $\g$ over $\C$  \cite{Kac77}. These are divided into three groups namely: basic Lie superalgebras (which means the classical and exceptional series), the strange ones (often also called periplectic and queer) and the ones of Cartan type. Kac also classified the simple finite dimensional representations of the  basic classical Lie superalgebras in \cite{Kac77a, Kac77, Kac78}. In recent times there  has been much interest in understanding finite dimensional modules for the Lie superalgebra $\g\otimes A$ where $\g$ is simple finite dimensional Lie superalgebra and $A$ is commutative associative algebra with unit over complex numbers $\C$.  For example, If we take $A = \C[X]$, then the Lie superalgebra $\g\otimes \C[X]$ is called a current superalgebra. If we take $A = \C[X, X^{-1}]$, then $\g\otimes \C[X, X^{-1}]$ is called a loop superalgebra. If we take $A = \C[X_1^{\pm 1},\ldots, X_n^{\pm 1}]$, then  $ \g \otimes \C[X_1^{\pm 1},\ldots, X_n^{\pm 1}]$ is called a multiloop superalgebra. The classification of finite dimensional irreducible modules for multiloop superalgebra is obtained in \cite{Rao13, RZ04}. In more general setting, the irreducible finite dimensional modules were classified in \cite{Sav14, CMS16}. 

The Weyl modules play important role in the representation theory of infinite-dimensional Lie algebras. Chari and Pressley \cite{CP01} introduced Weyl modules (global and local) for the loop algebra $\g\otimes \C[X^{\pm 1}]$, where $\g$ is simple Lie algebra over $\C$ and proved that these modules are indexed by dominant integral weights of $\g$ and are closely related to certain irreducible modules for quantum affine algebras.  Feigin and Loktev \cite{FL04} extended the notion of Weyl modules to the higher- dimensional case, i.e., instead of the loop algebra they worked with the Lie algebra $\g \otimes A$ where $A$ is the coordinate ring of an algebraic variety and obtained analogs of some of the results of \cite{CP01}. Later in \cite{CFK10}, Chari et. al., consider a more general functorial approach to Weyl modules associated to the algebra $\g \otimes A$ where $A$ is commutative associative unital algebra over $\C$. Also twisted versions of Weyl modules have been defined and investigated in \cite{CFS08, FMS13}. The Weyl modules for equivariant map algebras has been studied in \cite{FMS15}. 

However, in super setting the study of Weyl modules has less developed than the corresponding theory in Lie algebras. At first Zhang in \cite{Zha14}, define and study the Weyl modules in the spirit of Chari-Pressley for a quantum analogue in the loop case for $\g = \mathfrak{sl}(m, n)$. In \cite{CLS19}, Calixto, Lemay and Savage  study Weyl modules for Lie superalgebras of the form $\g\otimes_{\C} A$, where A is an associative commutative unital $\C$-algebra and $\g$ is a classical Lie superalgebra or $\mathfrak{sl}(n, n), n \geq 2$. Particularly, they define Weyl modules (global and local) for the Lie superalgebras $\g\otimes_{\C} A$ and prove that global Weyl modules are universal highest weight objects in a certain category and local Weyl modules are finite dimensional. Furthermore recently, Bagci, Calixto and Macedo \cite{BCM19} study Weyl modules (global and local) and Weyl functors for the superalgebras $\g\otimes A$, where $\g$ is either $\mathfrak{sl}(n, n),\ n\geq 2$, or any finite dimensional simple Lie superalgebra not of type $\q(n)$, and $A$ is an associative, commutative algebra with unit.

The goal of this paper is to study global and local Weyl modules for Lie superalgebras $\g\otimes_{\C} A$, where A is an associative commutative unital $\C$-algebra and $\g$ is the {\em queer Lie superalgebra}. To prove our results, we follow \cite{CLS19}.
%

%
\section{Preliminaries}\label{sec:prelim}
%

Throughout the paper ground field  will be the field of complex numbers $\C$. By $\Z_{\geq 0}$ and $\Z_{>0}$ we denote the nonnegative integers and strictly positive integers, respectively. Also we set $\Z_2 = \Z/2\Z$. All supervectorspaces, superalgebras, tensor products etc. are  defined over  $\C$. In this section, we review some facts about associative commutative algebras and queer Lie superalgebras that we need in the sequel.  

\subsection{Basic definitions}


A vector space $V$ is called a {\em supervectorspace} if $V$ is $\Z_2$-graded, i.e.,  $V= V_{\bar{0}} \oplus V_{\bar{1}}$. The dimension of the vector space $V$ is the tuple $(\dim V_{\bar{0}} \mid \dim V_{\bar{1}})$. The parity of a homogeneous element $v \in V_i$ is denoted by $|v| = i, i \in \Z_2$. An element in $V_{\bar{0}}$ is called even, while an element in $V_{\bar{1}}$ is called odd. A {\em subspace} of $V$ is a $\Z_2$-graded vector space $W = W_{\bar{0}} \oplus W_{\bar{1}} \subseteq V$ with compatible $\Z_2$-gration, i.e., $W_{ i} \subseteq V_{ i}$, for $i \in \Z_2$.  We denote by $\C^{m\mid n}$ the supervectorspace $\C^{m}\oplus \C^n$, where the first summand is even and the second summand is odd. 

Given two supervectorspaces $V$ and $W$, a linear mapping $T: V \longrightarrow W$ is {\em homogeneous of degree} $d\in \Z_2$ if $T(V)_{ i} \subset W_{i +{d}}$ for $i \in \Z_2$. The map $T$ is called even (respectively, odd) if ${d} = \bar{0}$ (respectively, ${d} = \bar{1}$). Consider the vector space of all linear transformations from $V$ to $W$ denoted as $\Hom(V, W)$ is $\Z_2$-graded with
\[\Hom(V, W)_d = \{T: V \longrightarrow W\mid T \;\mbox{is homogeneous of degree}\; d\},\]
where $d\in \Z_2$ . Define $ \End(V) := \Hom(V, V)$. The supervectorspaces and homogeneous mappings define a category. If we restrict the mappings to homogeneous even mappings we obtain an abelian category say {\em Vec}.
We denote by $\Pi$ the {\em parity change functor}, on category {\em Vec} which is defined as 
\[\Pi (V) = \Pi (V)_{\bar{0}} \oplus \Pi (V)_{\bar{1}}, \quad \Pi (V)_{i}= V_{i + \bar{1}}, \;\;   i \in \Z_2 \] and  $\Pi f=f$ for  $V \in  \mbox{Vec}$ and  $f: V \longrightarrow W \in  \mbox{Vec}$. For example consider if $V$ is of dimension $(1 \mid 0)$ then $\Pi V$ has dimension $(0 \mid 1)$. We assume the field $\mathbb{C}$ is of homogeneous even dimension.

A supervectorspace $ \mathcal{A}= A_{\bar{0}} \oplus A_{\bar{1}}$, equipped with a bilinear associative multiplication satisfying $A_{i}A_{j}\subseteq A_{i+j},$ for $i, j\in \Z_2$ is called a $\Z_2$-graded associative algebra or, {\em associative superalgebra}. For instance $\End(V)$ is an associative superalgebra. A homomorphism between two superalgebras $\mathcal{A}$ and $\mathcal{B}$ i.e., $f: \mathcal{A} \longrightarrow \mathcal{B}$, is a even linear map($f(\mathcal{A}_{i}) \subseteq \mathcal{B}_{i}$ for $ i \in \mathbb{Z}_{2}$) with $f(ab)=f(a)f(b)$. The tensor product $\mathcal{A}\otimes \mathcal{B}$ is a superalgebra, with underlying vector space is the tensor product of supervectorspaces of $\mathcal{A}$ and $\mathcal{B}$, with the induced $\Z_2$-grading and multiplication  is given by $(a_1 \otimes b_1) (a_2 \otimes b_2) = (-1)^{|a_2||b_1|}a_1a_2 \otimes b_1b_2$ for homogeneous elements $a_i\in \mathcal{A}$ and $b_i\in \mathcal{B}$. A {\em module} $M$ over a superalgebra $\mathcal{A}$ is always understood in the  $\Z_2$-graded sense, that is $M= M_{\bar{0}} \oplus M_{\bar{1}}$ such that $A_iM_j \subseteq M_{i+j}$, for $i, j \in \Z_2$. Subalgebras and ideals of superalgebras are $\Z_2$-graded subalgebras and ideals. A superalgebra that has no non-trivial ideal is called {\em simple}. A homomorphism between $\mathcal{A}$-modules $M$ and $N$ is an even linear map $f: M\longrightarrow N$( i.e., $f(M_{i}) \subseteq N_{i}$ for $i \in \mathbb{Z}_{2}$), with $f(am) = a f(m)$, for all $a \in \mathcal{A}, m \in M$.

\subsection{Lie superalgebras}

\begin{defn}[Lie superalgebra]
  A \emph{Lie superalgebra} is a $\Z_2$-graded vector space $\g=\g_{\bar 0}\oplus \g_{\bar 1}$ with a bilinear multiplication $[\cdot,\cdot]$ satisfying the following axioms:
  \begin{enumerate}
    \item The multiplication respects the grading: $[\g_i,\g_j] \subseteq \g_{i+j}$ for all $i,j \in \Z_2$.
    \item Skew-supersymmetry: $[a,b]=-(-1)^{|a||b|}[b,a]$, for all homogeneous elements $a,b\in \g$.
    \item Super Jacobi Identity: $[a,[b,c]]=[[a,b],c]+(-1)^{|a||b|}[b,[a,c]]$, for all homogeneous elements $a,b,c\in \g$.
  \end{enumerate}
\end{defn}
\begin{ex}
Let A be any associative superalgebra. Then we can make $A$ into a Lie superalgebra by defining $[a, b]:=a b -(-1)^{|a||b|}b a$ for all homogeneous elements $a, b \in A$ and extending $[., .]$ by linearity. We call this is the Lie superalgebra associated with $A$. A concrete example is the general linear Lie superalgebra $\mathfrak{g l}(V)$ associated with associative superalgebra $End(V)$ of all linear operators on a $Z_{2}$-graded vectorspace $V$.
\end{ex}
A homomorphism $\rho$ between  Lie superalgebras is a map which preserves the structure in them. Precisely $\rho: \g \longrightarrow \g_{1}$ is an even linear map with $\rho([x, y])=[\rho x, \rho y]$ for all $x, y \in \g$.
\begin{defn}
A representaion of Lie superalgebra $\g$ is a Lie superalgebra homomorphism $\rho: \g \longrightarrow \mathfrak{gl}(V)$,i.e., $\rho$ is an even linear with $ \rho[x, y]=\rho(x) \rho(y)- (-1)^{ |x| |y|} \rho(y) \rho(x) $.
\end{defn}
 Alternatively $V$ is called $\g$-module and $V$ is irreducible if there are no submodule other than $0$ and $V$ itself.
\begin{lemma}\cite{Sav14}
Suppose $\g$ is a Lie superalgebra and $V$ is an irreducible $\g$-module such that $Iv=0$ for some ideal $I$ of $\g$ and non-zero vector $v\in V$. Then $IV=0$.
\end{lemma}

Given a Lie superalgebra $\g$, we will denote by $\mathbf{U}(\g)$ its {\em universal enveloping superalgebra}.  The universal enveloping superalgebra $\mathbf{U}(\g)$ is constructed from the tensor algebra $T(\g)$ by factoring out  the ideal generated by the elements $[u, v] - u\otimes v + (-1)^{|u||v|} v\otimes u$, for homgeneous elements $u, v$ in $\g$. Now we state an analogous of PBW Theorem in super setting, which ensures that $\g \mapsto \mathbf{U}(\g)$ is an inclusion by precisely giving a basis for $\mathbf{U}(\g)$.
\begin{lemma}[\cite{Mus12}, Theorem 6.1.1]\label{lem3}
Let $\g = \g_{\bar0} \oplus \g_{\bar1}$ be a Lie superalgebra. If $x_1, \ldots, x_m$ be a basis of $\g_{\bar0}$ and $y_1, \ldots, y_n$ be a basis of $\g_{\bar1}$, then the monomials
\[x_1^{a_1}\cdots x_m^{a_m} y_1^{b_1}\cdots y_n^{b_n}, \quad a_1, \ldots, a_m \geq 0, \quad \mbox{and}\quad b_1, \ldots, b_n \in \{0, 1\},\]
form a basis of $\mathbf{U}(\g)$. In particular, if $\g$ is finite dimensional and $\g_{\bar 0} = 0$, then $\mathbf{U}(\g)$ is finite dimensional.
\end{lemma}

\subsection{The queer Lie superalgebra}

Let $V= V_{\bar{0}} \oplus V_{\bar{1}}$ be a supervectorspace with $\dim V_{\bar 0} = \dim V_{\bar 1}$. Choose $P\in \mbox{End}(V)_{\bar 1}$ such that $P^2 =- 1$.  The subspace
\[{\q}(V) =\{T \in \mbox{End}(V) \mid [T, P]=0\}\]
is a subalgebra of $\mathfrak{gl}(V )$ called the, queer Lie superalgebra. If $V = \C^{n\mid n}$, then with a homogeneous basis we identify $\mathfrak{gl}(V)$ with $\mathfrak {gl}(n \mid n)$.

Now for explicit realization of the queer Lie superalgebra $\q(n)$ in terms of matrices, set 
\begin{equation}\label{eq2}
P:=   \left(\begin{array}{@{}c|c@{}}
    0 & I_n \\ \hline
    -I_n & 0
  \end{array}\right).
  \end{equation}
Then, for $X \in \mathfrak{gl}(n \mid n)$, we have $X \in \q(n)$ if and only if
$XP-(-1)^{|X|}PX = 0$ holds. Hence $\q(n)$ consisting of matrices of the form 
  \begin{equation}\label{eq3}
  \left(\begin{array}{@{}c|c@{}}
    A & B \\ \hline
    B & A
  \end{array}\right)
  \end{equation} 
  where $A$ and $B$ arbitrary  $n\times n$ matrices with \begin{equation}
\q(n)_{\bar 0} = \left\{  \left(\begin{array}{@{}c|c@{}}
    A & 0 \\ \hline
    0 & A
  \end{array}\right)\mid A \in \gl(n)\right\} \quad
   \q(n)_{\bar 1} = \left\{  \left(\begin{array}{@{}c|c@{}}
    0 & B \\ \hline
    B & 0
  \end{array}\right)  \mid  B \in \gl(n)\right\}.
\end{equation}
 From now on we denote $\q(n)=: \q$.


 A subalgebra of $\q$ is called a Cartan subalgebra if it is a self-normalizing nilpotent subalgebra. Every such subalgebra has a non-trivial odd part. 
Denote by $N^{-}, H, N^{+}$ respectively the strictly lower triangular, diagonal and strictly upper triangular matrices in $\mathfrak{gl}(n)$. Then we define
\begin{equation}\label{eq5}
\h_{\bar 0} = \left\{  \left(\begin{array}{@{}c|c@{}}
    A & 0 \\ \hline
    0 & A
  \end{array}\right)\mid A \in H\right\} \quad
   \h_{\bar 1} = \left\{  \left(\begin{array}{@{}c|c@{}}
    0 & B \\ \hline
    B & 0
  \end{array}\right)  \mid  B \in H\right\},
\end{equation}

\begin{equation}\label{eq6}
\n_{\bar 0}^{\pm} = \left\{  \left(\begin{array}{@{}c|c@{}}
    A & 0 \\ \hline
    0 & A
  \end{array}\right)\mid A \in N^{\pm}\right\} \quad
   \n_{\bar 1}^{\pm} = \left\{  \left(\begin{array}{@{}c|c@{}}
    0 & B \\ \hline
    B & 0
  \end{array}\right)\mid B \in N^{\pm}\right\},
\end{equation}
\begin{equation}\label{eq7}
\h = \h_{\bar 0}\oplus \h_{\bar 1}\quad \mbox{and}\quad \n^{\pm} = \n_{\bar 0}^{\pm} \oplus \n_{\bar 1}^{\pm}.
\end{equation}

\begin{lemma}[\cite{Mus12}, Lemma 2.4.1]\label{lem4}
We have a vector space decomposition
\[\q = \n^{-}\oplus \h \oplus \n^{+}\]
such that $\n^{-}, \n^{+}, \h$ are graded subalgebra of $\q$ with $\n^{\pm}$ nilpotent. The subalgebra $\h$ is called the {\em standard Cartan subalgebra} of $\q$. 
\end{lemma}
Given any Lie superalgebra $\mathfrak{a}$, the map 
$x\longrightarrow x \otimes 1 \oplus 1 \otimes x, x \in \mathfrak{a}$ extends to an algebra homomorphism $\mathbf{U}(\mathfrak{a}) \longrightarrow \mathbf{U}(\mathfrak{a})  \otimes \mathbf{U}(\mathfrak{a})$.  By the PBW Theorem (see Lemma \ref{lem3}), we know that if $\mathfrak{b}$ and $\mathfrak{c}$ are subalgebras of $\mathfrak{a}$ such that $\mathfrak{a} = \mathfrak{b}\oplus \mathfrak{c}$ as vector spaces 
\[\mathbf{U}(\mathfrak{a})\cong \mathbf{U}(\mathfrak{b})\otimes \mathbf{U}(\mathfrak{c}).\]
Thus, from Lemma \ref{lem4}, we obtain the triangular decomposition of $\mathbf{U}(\q)$:
\begin{equation}\label{eq8}
\mathbf{U}(\q)\cong \mathbf{U}(\n^-)\otimes \mathbf{U}(\h) \otimes \mathbf{U}(\n^+).
\end{equation}

\subsection{Root system for $\q$.}
We fix $\h = \h_{\bar 0}\oplus \h_{\bar 1}$ to be standard Cartan subalgebra of $q$, is given by 
\[\h_{\bar 0} = \C k_1\oplus \cdots \oplus \C k_{n} \quad \mbox{and}\quad \h_{\bar 1} = \C k_1'\oplus \cdots \oplus \C k_{n}'\]
where
\[k_i =  \left(\begin{array}{@{}c|c@{}}
     E_{i,i} & 0 \\ \hline
    0 & E_{i,i} 
  \end{array}\right)\quad \mbox{and}\quad k_i' = \left(\begin{array}{@{}c|c@{}}
     0 &E_{i,i}  \\ \hline
     E_{i,i} &0 
  \end{array}\right)\]
and $E_{i, j}$ is the $n\times n$ matrix having $1$ at the $(i, j)$-entry and $0$ elsewhere. The Cartan subalgebra $\h$ has a nontrivial odd part $\h_{\bar 1}$ and hence is not abelian, as $[\h_{\bar 0}, \h] = 0$ and $[\h_{\bar 1}, \h_{\bar 1}] = \h_{\bar 0}$. Note that all Cartan subalgebra of $\q$ are conjugate to $\h$. For $1\leq i \neq j \leq n$, we set 
\[e_{i, j} = \left(\begin{array}{@{}c|c@{}}
     E_{i,j} & 0 \\ \hline
    0 & E_{i,j} 
  \end{array}\right)\quad \mbox{and}\quad e_{i, j}' = \left(\begin{array}{@{}c|c@{}}
     0 &E_{i,j}  \\ \hline
     E_{i,j} &0 
  \end{array}\right).\]
  The set $\{e_{i, j}, e_{i,j}^{'} \mid 1 \leq i, j \leq n\}$ is a homogeneous linear basis for $\q$. The even subalgebra is $\q_{\bar{0}}$ is spanned by $\{e_{i, j}  \mid 1 \leq i, j \leq n\}$ and hence is isomorphic to the general linear Lie algebra $ \mathfrak{gl}(n)$ and odd space $\q_{\bar{1}}$ is is isomorphic to the adjoint module.
  Let $\{\epsilon_1, \ldots, \epsilon_{n}\}$ be the basis of $\h_{\bar 0}^{*}$
 dual to $\{k_{1}, \cdots, k_{n}\}$ defined as $\epsilon_{i}( \left(\begin{array}{@{}c|c@{}}
     h & 0 \\ \hline
    0 & h
  \end{array}\right))=a_{i}$, for any diagonal matrix $h$ with diagonal entries $(a_{1}, a_{2}, \cdots, a_{n})$. We denote $h_{i}:=k_{i}-k_{i+1}$ for $1 \leq i \leq n-1$. Given $\alpha \in \h_{\bar 0}^{*}$, let 
  \[\q_{\alpha} = \{x\in \q\mid [h, x] = \alpha(h) x\;\; \mbox{for all}\; h \in \h_{\bar 0}\}.\]
 Note that $\q_0 = \h$. We call $\alpha \neq 0$ a root if $\q_{\alpha} \neq 0$. The set $\Phi=\{\alpha| q_{\alpha} \neq 0\}$ is called the root system of $\q$.  A root $\alpha$ is called even root if $\q_{\alpha}\cap \q_{\bar{0}}\neq0$ and it is called odd if $\q_{\alpha}\cap \q_{\bar{1}}\neq 0$. The root system  $\Phi=\Phi_{\bar{0}} \cup \Phi_{\bar{1}}$ of $\q$ has identical even and odd parts where $\Phi_{\bar{0}}$ denote the set of even roots and $\Phi_{\bar{1}}$ denote the set of odd roots. Namely $\Phi_{\bar{0}} = \Phi_{\bar{1}}=\{ \epsilon_{i} - \epsilon_{j} \mid 1 < i\neq j<n\}$.  For  each root $\alpha = \epsilon_i -\epsilon_{j},\; 1\leq i \neq j \leq n$, we have  root spaces has dimension $(1\mid 1)$, 
 \[\q_{\alpha} = \C e_{i,j} \oplus \C e_{i,j}'.\]  
 and 
 \[\q =  \bigoplus_{\alpha \in \h_{\bar 0}^{*}} \q_{\alpha}.\] is the root space decomposition of $\q$. 
 \smallskip

A root $\alpha$ is called {\em positive} (resp. {\em negative}) if $\q_{\alpha}\cap \n^+ \neq 0$ (resp. $\q_{\alpha}\cap \n^- \neq 0$). We denote by $\Phi^{+}$ (resp. $\Phi^{-}$) the subset of positive (resp. negative) roots. Denote by $\Delta$ the set of simple roots. Thus, 
\[ \Phi^+ = \{\epsilon_i -\epsilon_j \mid 1\leq i< j\leq n\},\;\Phi^- = -\Phi^+,\; \Phi = \Phi^+ \cup \Phi^-, \;\Delta = \{\epsilon_i - \epsilon_{i+1}\mid 1\leq i \leq n-1\}.\]
Hence,
\[\n^+ = \bigoplus_{\alpha\in \Phi^+} \q_{\alpha}\quad \mbox{and}\quad \n^- = \bigoplus_{\alpha\in \Phi^-} \q_{\alpha}.\]A maximal solvable subalgebra of $\q$ is called Borel subalgebra $\bb$. Borel subalgebra of $\q$ is conjugate to the standard Borel subalgebra $\bb_{+} = \h \oplus \n^+$ of $\q$. 
Set $\alpha_{i}:= \epsilon_{i}-\epsilon_{i+1}$ and the root space $\q_{\alpha_{i}}$ is spanned by \[e_i :=  \left(\begin{array}{@{}c|c@{}}
     E_{i,i+1} & 0 \\ \hline
    0 & E_{i,i+1} 
  \end{array}\right)\quad \mbox{and}\quad e_i' = \left(\begin{array}{@{}c|c@{}}
     0 &E_{i,i+1}  \\ \hline
     E_{i,i+1} &0 
  \end{array}\right),\]
 while  $\q_{-\alpha_{i}}$ is spanned by \[f_i =  \left(\begin{array}{@{}c|c@{}}
     E_{i+1,i} & 0 \\ \hline
    0 & E_{i,i+1} 
  \end{array}\right)\quad \mbox{and}\quad f_i' = \left(\begin{array}{@{}c|c@{}}
     0 &E_{i+1,i}  \\ \hline
     E_{i+1,i} &0 
  \end{array}\right).\]

Hence $n^{+}$ is spanned by $e_{i}, e_{i}'$
and $n^{-}$ is spanned by $f_{i}, f_{i}'$ and the standard Borel is spanned by $e_{i}, e_{i}', h_{i}, k_{i}'$.
\\
For $\alpha=\epsilon_{i}-\epsilon_{j}\in  \Phi_{\bar 0}^{+}$, let $s_{\alpha}: \h_{\bar 0}^{*} \longrightarrow  \h_{\bar 0}^{*}$ be the corresponding reflection and is defined by \[s_{\epsilon_i-\epsilon_j}(\epsilon_i) = \epsilon_{j},\quad  s_{\epsilon_i-\epsilon_j}(\epsilon_k) = \epsilon_{k},\ \mbox{for}\; k \neq i, j.\] The Weyl group of $\q$ is the Weyl group $W$ of $\q_{\bar 0}$ generated by $s_{\alpha}$ where $\alpha\in \Phi_{\bar 0}^{+}$which is the symmetric group $\mathfrak{S}_{n}$ in $n$ letters.

Let $I := \{1, 2, \ldots, n-1\}$ and $J:= \{1, 2, \ldots, n\}$.
\begin{prop}\cite{GJKM10}
The Lie superalgebra $\q$ generated by the elements $e_{i}, e_{i}', f_{i}, f_{i}'$ for $i\in I$, $\h_{\bar{0}}$ and $k_{j}'$ for $j\in J$  with the following  defining relations
\begin{align*}
&[h, h']=0 ~~for h, h' \in \h_{\bar{0}},\\
&[h, e_{i}]=\alpha_{i}(h)e_{i}, [h, e_{i}']=\alpha_{i}(h)e_{i}'\;\, \quad \mbox{for $h\in \h_{\bar{0}},i \in I$},\\
&[h, f_{i}]=-\alpha_{i}(h)f_{i}, [h, f_{i}']=-\alpha_{i}(h)f_{i}'\;\; \quad \mbox{for $ h\in \h_{\bar{0}},i\in I$},\\
&[h, k_{l}']=0, \;\; \mbox{for $h \in \h_{\bar{0}}\,\; l \in J$} ,\\
&[e_{i}, f_{j}]= \delta_{i j}(k_{i}-k_{i+1}), [e_{i}, f_{j}']=\delta_{i j}(k_{i}'-k_{i+1}') \mbox{for $i, j\in I$},\\
&[e_{i}', f_{j}]=\delta_{i j}(k_{i}'-k_{i+1}'), [k_{l}', e_{i}]=\alpha_{i}(k_{l}) e_{i}'\mbox{ for $i, j\in I, l\in J$},\\
 &[k_{l}', f_{i}]=-\alpha_{i}(k_{l})f_{i}', [e_{i}', f_{j}']= \delta_{i j}(k_{i}+k_{i+1}), \mbox{for $i, j\in I, l\in J$},\\
  &[k_{l}',e_{i}']= \begin{cases} e_i& \mbox{if $l =i, i+1$}\\
  0& \mbox{otherwise}
  \end{cases}\;\;\; \mbox{for $ i\in I, j \in J$},\\ 
  &[k_{l}',f_{i}']= \begin{cases} f_i& \mbox{if $l =i, i+1$}\\
  0& \mbox{otherwise}
  \end{cases}\;\;\; \mbox{for $i\in I, j \in J$},\\ 
  &[e_{i},e_{j}']=[e_{i}', e_{j}']= [f_{i},f_{j}']=[f_{i}',f_{j}']=0\;\; \mbox{for $1 \leq i, j \leq n-1, |i-j| \neq1$}\\
  &[e_{i},e_{j}]=[f_{i},f_{j}]=0, \mbox{for $i, j \in I, |i-j| > 1$},\\
  &[e_{i},e_{i+1}]=[e_{i}',e_{i+1}'], [e_{i},e_{i+1}']=[e_{i}',e_{i+1}],\\
  &[f_{i+1},f_{i}]=[f_{i+1}',f_{i}'], [f_{i+1},f_{i}']=[f_{i+1}',f_{i}],\\
  &[k_{i}',k_{j}']=\delta_{ij}2k_i \;\;\;\mbox{for $ j \in J$}\\
  &[e_{i},[e_{i}, e_j]]=[e_{i}',[ e_{i}, e_j]]= 0\;\;  \mbox{for $ i, j\in I, |i-j| =1$}\\
&[f_{i},[f_{i}, f_j]]=[f_{i}',[f_{i}, f_j]]= 0\;\;  \mbox{for $ i, j\in I, |i-j| =1$}.\\
 \end{align*}
\end{prop}

The simple roots $\Delta$ of $\q$ satisfy the following property:
\begin{equation}\label{eqsystem30}
\mbox{For all $\alpha \in \Delta_{\bar{1}}$, there exists $\alpha'\in \Phi_{\bar{1}}^+$ such that $\alpha+\alpha'\in \Phi$}.
\end{equation}
This is true, as every root of $\q$ is even as well as odd.

 
\subsection{Clifford Algebra.}

\begin{defn}[Clifford Algebra]
Let $V$ be a finite dimensional vector space and $f:V\times V\to \C$ be a symmetric bilinear form. We call the pair $(V, f)$ a quadratic pair. Let $I$ be the ideal of the tensor algebra $T(V)$ generated by the elements 
\[x\otimes x - f(x, x)1, \quad x\in V\] and set $\mbox{Cliff}(V, f) = T(V)/I$. The algebra $\mbox{Cliff}(V, f)$ is alled the Clifford algebra of the pair $(V, f)$ over $\C$.
\end{defn}
\begin{rmk}[\cite{Hus94}, Ch. 12, Def. 4.1 and Theorem 4.2]
For a quadratic pair $(V, f)$, there exists a linear map $\theta: V\to \mbox{Cliff}(V, f)$ such that the pair $(\mbox{Cliff}(V, f), \theta)$ has the following universal property: For all linear maps $\eta: V\to A$ such that $\eta(v)^2 = f(v, v)1_A$ for all $v\in V$, where $A$ is a unital algebra, there exists a unique algebra homomorphism $\eta': \mbox{Cliff}(V, f) \to A$ such that $\eta' \circ \theta = \eta$, in other words, we have the following commutative diagram. 
\begin{center}
$\xymatrix{
V \ar[rd]_{\eta} \ar[rr]^{\theta} & & \mbox{Cliff}(V, f) \ar[ld]^{\eta'} \\
& A &} $ 
\end{center}
\end{rmk}

Clifford algebra have a natural superalgebra structure. In fact, $T(V)$ possess a $\Z_2$-grading such that $I$ is homogeneous, so the grading descends to $\mbox{Cliff}(V, f)$. Thus resulting superalgbera $\mbox{Cliff}(V, f)$ is sometimes called the Clifford superalgebra. When $f$ is known from the context, we shall write $\mbox{Cliff}(V)$ instead of $\mbox{Cliff}(V, f)$.

 For $\lambda\in \h_{\bar 0}^{*}$, define an even super antisymmetric bilinear form $F_{\lambda}$ on $\h_{\bar 1}$, by setting $F_{\lambda}(u, v) = \lambda ([u, v])$ and denote $E_{\lambda} := \h_{\bar 1}/\ker F_{\lambda}$. Let $\mbox{Cliff}(\lambda)$ be the Clifford superalgebra with respect to quadratic pair $(E_{\lambda}, F_{\lambda})$ and $\mbox{Cliff}(\lambda)$  is endowed with a canonical $\Z_2$-grading. By definition we have an isomorphism of superalgebras 
 \begin{equation}\label{cliffeq29}
 \mbox{Cliff}(\lambda) \cong U(\h)/I_{\lambda},
 \end{equation} where $I_{\lambda}$ denoted the ideal of $U(\h)$ generated by $\ker F_{\lambda}$ and $a-\lambda(a)$ for $a\in  \h_{\bar 0}$. 
 
Let $\h_{\bar 1}'\subseteq \h_{\bar 1}$ be a maximal isotropic subspace with respect to $F_{\lambda}$ and define the Lie superalgebra $\h' := \h_{\bar 0}\oplus \h_{\bar 1}'$. Let  $\C v_{\lambda}$, be the one-dimensional $\h_{\bar 0}$-module defined by $hv_{\lambda} = \lambda(h)v_{\lambda}$ for all $h \in \h_{\bar{0}}$, extends to an $\h'$-module by setting $\h_{\bar 1}' v_{\lambda} = 0$. Then the induced module  
\[ \mbox{Ind}_{\h'}^{\h}\C v_{\lambda}=\h \otimes \C v_{\lambda}\] is an irreducible $\h$-module. If $\mbox{Ind}_{\h'}^{\h}\C v_{\lambda}$ is a finite dimensional irreducible $\h$-module, then $\mbox{Ind}_{\h'}^{\h}\C v_{\lambda}$ is a finite dimensional irreducible module over $\mbox{Cliff}(\lambda)$ via the pullback through \eqref{cliffeq29}.

We may consider $\mbox{Cliff}(\lambda)$ as the associative $\C$-algebra generated by the identity ${\bf 1} = 1+ I_{\lambda}$ and $t_{\bar{i}}:= k_{\overline{i}}+ I_{\lambda}$ satisfying the relations 
\begin{equation}
t_{\bar{i}}t_{\bar{j}}+ t_{\bar{j}}t_{\bar{i}} = 2\delta_{ij}, \quad i, j =1, 2, \ldots, n.
\end{equation}
Let $S = \oplus_{i=1}^n \C t_{\bar{i}}$ and $\lambda= (\lambda_1, \ldots, \lambda_n)\in \C^n$ and denote by $B_{\lambda}:S\times S\to \C$ the symmetric bilinear form defined by $B_{\lambda} (t_{\bar{i}},t_{\bar{j}}) = \delta_{ij}\lambda_i$. Let $\mbox{Cliff}_S(\lambda)$ be the unique up to isomorphism Clifford algebra associated to $S$ and $B_{\lambda}$. Now define $S(\lambda):= S/\ker B_{\lambda}$ and denote by $\beta_{\lambda}$ the restriction of $B_{\lambda}$ on $S(\lambda)$. Let $N_{\lambda} = \{i\mid \lambda_i\neq 0\}, Z_{\lambda}= \{j\mid \lambda_j=0\}$ and $\ell = \#N_{\lambda}$. Set 
\[\lambda_N:= (\lambda_{i_1}, \ldots, \lambda_{i_{\ell}}), 0_Z := (\lambda_{j_1}, \ldots, \lambda_{j_{n-\ell}})= (0, \ldots, 0).\]One can see that $\ker B_{\lambda} = \oplus_{j\in Z_{\lambda}}\C t_{\overline{j}}$ and $\mbox{Cliff}_S(\lambda_{N}) = \oplus_{i\in N_{\lambda}} \C t_{\overline{i}}$ is the Clifford algebra corresponding to $(S(\lambda), \beta_{\lambda})$. Further, 
\begin{equation}\label{cliffeq12}
\mbox{Cliff}_S(\lambda) \cong \mbox{Cliff}_S(\lambda_N) \otimes_{\C} \mbox{Cliff}_S(0_Z) \cong \mbox{Cliff}_S(\lambda_N) \otimes_{\C} \bigwedge \ker B_{\lambda}.
\end{equation}
Here $\bigwedge U$ denotes the exterior algebra of the vector space $U$. Thus by the isomorphisms in \eqref{cliffeq12}, every $\mbox{Cliff}_S(\lambda)$-module can be considered as a  $\mbox{Cliff}_S(\lambda_N)$-module under the embedding \[\mbox{Cliff}_S(\lambda_N)  = \mbox{Cliff}_S(\lambda_N) \otimes_{\C} 1 \to \mbox{Cliff}_S(\lambda_N) \otimes_{\C} \mbox{Cliff}_S(0_Z).\] Then one can easily prove the following lemma.
\begin{lemma}\label{clifflem1}
Let $M$ be an irreducible $\mbox{Cliff}_S(\lambda)$ module. Then  $M$ is an irreducible $\mbox{Cliff}_S(\lambda_{N})$-module and $t_{\overline{i}}v =0$ for every $i\in Z_{N}$.
\end{lemma}

\subsection{Highest weight modules over $\q$.}

From now on, for a superalgebra $\mathcal{A}$, an $\mathcal{A}$-module will be understood as an $\mathcal{A}$-supermodule. A $\q$-module $M$ is called a weight module if it admits a weight space decomposition 
\[M = \bigoplus_{\mu \in \h_{\bar 0}^{*}} M_{\mu}, \;\;\mbox{where}\;\; M_{\mu} = \{m \in M \mid hm = \mu(h)m \; \mbox{for all}\; h \in \h_{\bar 0}\}. \]
An element $\mu \in \h_{\bar 0}^{*}$ such that $M_{\mu} \neq 0$ is called a {\em weight} of $M$ and $M_{\mu}$ is called weight space. The set of all weights of $M$ is denoted by $\wt(M)$.
\begin{defn}
A weight module $M$ is called a {\em highest weight module}  with highest weight $\lambda\in \h_{\bar 0}^{*}$ if $M_{\lambda}$ is finite dimensional and satisfies the following conditions:
\begin{enumerate}
\item $M$ is generated by $M_{\lambda}$,
\item $e_{i} v = e_{i}' v =0$ for all $v\in M_{\lambda}, \; i \in I$.
\end{enumerate}
\end{defn}

\begin{defn}
Let $\Lambda_{\bar 0}^+$ and  $\Lambda^+$ be the set of $\gl (n)$-dominant integral weights and the set of $\q$-dominant integral weights respectively, given by
\begin{align*}
&\Lambda_{\bar 0}^+ := \{\lambda_1\epsilon_1+ \cdots+\lambda_n\epsilon_n \in \h_{\bar 0}^{*}\mid \lambda_i-\lambda_{i+1}\in \Z_{\geq 0} \;\; \mbox{for all}\;\; i \in I\}\\
&\Lambda^+ := \{\lambda_1\epsilon_1+ \cdots+\lambda_n\epsilon_n \in \Lambda_{\bar 0}^+ \mid \lambda_i= \lambda_{i+1}\implies \lambda_i=\lambda_{i+1}\ = 0\;\; \mbox{for all}\;\; i \in I\}.
\end{align*}
\end{defn}

 \begin{prop}[\cite{Pen86}, Prop. 1]\label{prop1}
 Let $\mathbf{v}$ be a finite dimensional simple $\bb_{+}$-module:
 \begin{enumerate}
\item The maximal nilpotent subalgebra $\n^{+}$ of $\bb_{+}$ acts on $\mathbf{v}$ trivially.
\item There exists a unique weight $\lambda \in \h_{\bar 0}^{*}$ such that $\mathbf{v}$ is endowed with a canonical left $\mbox{Cliff}(\lambda)$-module structure and $\lambda$ determines $\mathbf{v}$ up to the parity reversing functor $\Pi$.
\item  For all $h \in \h_{\bar 0}, v \in \mathbf{v}$, we have $hv = \lambda(h)v$.
\end{enumerate}
 \end{prop}
 \begin{rmk}\label{remprop12}
 From Proposition \ref{prop1}, we know that the dimension of the highest weight space of a highest weight $\q$-module with highest weight $\lambda$ is the same as the dimension of an irreducible $\mbox{Cliff}(\lambda)$-module. On the other hand all irreducible $\mbox{Cliff}(\lambda)$-modules have the same dimension (see, for example, \cite[Table 2]{ABS64}). Thus the dimension of the highest weight space is constant for all highest weight modules with highest weight $\lambda$.
\end{rmk}
 
 \begin{prop}\label{prop4a}
Let $\lambda\in \Lambda^{+}$ and $V(\lambda)$ be the irreducible highest weight $\q$-module generated by an irreducible finite dimensional $\bb_{+}$-module {\bf v}. Then $ f_{i}^{\lambda(h_{i})+1}v=0$, for all $v\in {\bf v}$ and $i\in I$.
\end{prop}

\begin{proof}
Note that one can easily show by induction that for $k\in \Z_{\geq 0}$, 
\[e_{i}f_{i}^{k} = f_{i}^k e_{i} + kf_{i}^{k-1}(h_{i} - (k-1)).\]
Since $\n^{+}v = 0$ for all $v\in {\bf v}$, then
\begin{align*}
e_{i}f_{i}^{k} v &= f_{i}^k e_{i}(v)+ kf_{i}^{k-1}(h_{i} - (k-1)) v\\
& = (\lambda(h_{i}) - (k-1)) k f_{i}^{k-1}v.
\end{align*}
If $k = \lambda(h_{i}) +1$, one can see that $e_{i}f_{i}^{\lambda(h_{i})+1}v =0$. For $i\neq j$, as $[e_{j}',f_{i}]=0=[e_{j}, f_{i}]$, we have $e_{j}f_{i}^{\lambda(h_{i})+1}v = 0=e_j'f_{i}^{\lambda(h_{i})+1}v $. 

Now suppose that $e_i'f_{i}^{\lambda(h_{i})+1}v \neq 0$. Since $[e_{i}, e_i']=0$ for $|i-j| \neq 1$, so 
\[0=[e_{i}, e_i'] f_{i}^{\lambda(h_{i})+1}v = e_{i}(e_i'f_{i}^{\lambda(h_{i})+1}v) -e_i'(e_{i}f_{i}^{\lambda(h_{i})+1}v).\]  We get $e_{i}(e_i'f_{i}^{\lambda(h_{i})+1}v) = 0$ and  $e_i'(e_if_{i}^{\lambda(h_{i})+1}v) = 0$. Similarly,  as $[e_i', e_j'] = 0$ for $|i-j| \neq 1$, we get
\[ e_i'(e_i'f_{i}^{\lambda(h_{i})+1}v) =0.\]
Also, for $i\neq j$, we have 
\[e_{j}(e_i'f_{i}^{\lambda(h_{i})+1}v)  = e_j'e_i'f_{i}^{\lambda(h_{i})+1}v) = 0.\]
If $\lambda(h_{i})\geq 1$, then weight of the weight vector $e_i'f_{i}^{\lambda(h_{i})+1}v $is $ \lambda- \lambda(h_{i})\alpha_i <\lambda$. Thus, $e_i'f_{i}^{\lambda(h_{i})+1}v$ would generate a nontrivial proper submodule of $V(\lambda)$, which contradicts the irreducibility of $V(\lambda)$.
 
 If $\lambda(h_{i}) = 0$, then $\lambda_i=\lambda_{i+1}=0$ and since $v\in {\bf v}$, so by Lemma \ref{clifflem1}, we get $k_{i}' v= k_{i+1}' v =0$. Now
 \[e_i'f_iv = f_ie_i'v + (k_{i}' - k_{i+1}')v =0.\] Therefore, in any case $e_{i}'f_{i}^{\lambda(h_{i})+1}v = 0$. 
 
 Similarly, if $f_{i}^{\lambda(h_{i})+1}v \neq  0$, it would generate a non-trivial proper submodule of $V(\lambda)$. Hence, $f_{i}^{\lambda(h_{i})+1}v =0$ for all $v \in {\bf v}$.

\end{proof}
 \begin{defn}
 Let $\mathbf{v}(\lambda)$ be a finite dimensional irreducible $\bb_{+}$-module determined by $\lambda$ up to $\Pi$. The {\em Weyl module} $W(\lambda)$ of $\q$ with highest weight $\lambda$ is defined to be 
 \[W(\lambda) := \mathbf{U}(\q)\otimes_{\mathbf{U}(\bb_{+})} \mathbf{v}(\lambda).\]
  \end{defn}
  Note that in the above definition, the structure of $W(\lambda)$ is determined by $\lambda$ up to $\Pi$.

 \begin{prop}[\cite{Pen86}, Theorem 2, 4]\label{prop2}
\begin{enumerate}

\item For any weight $\lambda$, $W(\lambda)$ has a unique maximal submodule $N(\lambda)$.
\item For each finite dimensional simple $\q$-module $M$, there exists a unique weight $\lambda \in\Lambda_{\bar 0}^+$ and a surjective homomorphism $W(\lambda)\to M$ (one of the two possible $W(\lambda)$).
\item The irreducible quotient $L(\lambda):= W(\lambda)/ N(\lambda)$ is finite dimensional if and only if $\lambda \in \Lambda^+$.
\end{enumerate}
\end{prop}
\begin{rmk}
For $\lambda \in \Lambda^+$, up to isomorphism, there exists two simple finite dimensional modules w.r.t. highest weight $\lambda$ namely 
$ L(\lambda)$ and $\Pi L(\lambda)$, where $\Pi$ is the parity change functor. 
\end{rmk}


Let $P(\lambda) = \{\mu \in \h_{\bar 0}^{*}\mid M_{\mu}\neq 0\}$. Let $Q$ (resp.  $Q^+$) be the integer span (resp. $\Z_{>0}$-span ) of the simple roots. Denote by $\leq$ the usual partial order on $P(\lambda)$,
\[\mu_1, \mu_2 \in P(\lambda), \quad \mu_1 \leq \mu_2 \iff \mu_2-\mu_1 \in Q^+.\] 

Since $\q_{\bar 0} = \mathfrak{gl}(n)$ is reductive Lie algebra, for each even simple root $\alpha_{i}$ we can choose elements $e_{i} \in \q_{\alpha_{i}}, f_{i} \in \q_{-\alpha_{i}},$ and $h_{i}\in \h_{\bar{0}}$, such that the subalgebra generated by these elements is isomorphic to $\mathfrak{sl}(2)$, with these elements satisfying the relations for the standard Chevalley generators. In this case, we say that the set $\{e_{i}, f_{i}, h_{i}\}$ is an $\mathfrak{sl}(2)$-triple.

Denote the irreducible highest weight $\q$-module with highest weight $\lambda \in \h_{0}^{*}$, by  $L(\lambda)$ which is unique upto $\Pi$ and consider the weight space decomposition $L(\lambda) = \bigoplus_{\mu \in \h_{\bar 0}^{*}} L(\lambda)_{\mu}.$

\begin{defn} [The module $\bar{L}(\lambda)$]
For $\lambda\in \Lambda^+$, we define $\bar{L}(\lambda)$ (up to $\Pi$) to be the $\q$-module generated by $L(\lambda)_{\lambda}$ with defining relations
\begin{equation}\label{eq9}
\n^+k_{\lambda} = 0, \quad hk_{\lambda} = \lambda(h) k_{\lambda},\quad f_{i}^{\lambda(h_{i})+1}k_{\lambda} =0, \quad \mbox{for all}\quad k_{\lambda}\in L(\lambda)_{\lambda},  h\in \h_{\bar{0}},\;\;i\in I.
\end{equation}

 \end{defn}
\begin{prop}\label{prop3}
The module $\bar{L}(\lambda)$ is finite dimensional for all $\lambda\in \Lambda^+$.
\end{prop}
\begin{proof}
Let $x_1, \ldots, x_m$ and $y_1, \ldots, y_m$ be a homogeneous basis of $\q_{\bar0}$ and $\q_{\bar1}$, respectively. Then by Lemma \ref{lem3}, the monomials
\[x_1^{a_1}\cdots x_m^{a_m} y_1^{b_1}\cdots y_m^{b_m}, \quad a_1, \ldots, a_m \geq 0, \quad \mbox{and}\quad b_1, \ldots, b_m \in \{0, 1\},\]
form a basis of $\mathbf{U}(\q)$. Since $\{y_1^{b_1}\cdots y_m^{b_m} \mid b_j =0,1\}$ is a finite set, it is enough to show $\mathbf{U}(\q_{\bar 0})L(\lambda)_{\lambda}$ is finite dimensional. 

 Consider irreducible $\q_{\bar{0}}$-module  $V(\lambda)$ with highest weight $\lambda\in \Lambda^+.$ Since $\q_{\bar 0} = \mathfrak{gl}(n)$ is reductive Lie algebra,  $\lambda(h_{i})\in \N$ for each even simple root $\alpha_{i}$ with $i\in I$. Hence we have $V(\lambda)$ is finite dimensional. Note the centre $Z(\q_{\bar{0}})$ acts as a scalar on  $V(\lambda)$. Hence $V(\lambda)$ is isomorphic to $\q_{\bar 0}$-module generated by a vector $u_{\lambda}$ with defining relations
\[\n_{\bar 0}^+u_{\lambda} = 0, \quad hu_{\lambda} = \lambda(h) u_{\lambda},\quad f_{\alpha}^{\lambda(h_{i})+1}u_{\lambda} =0, \quad \mbox{for all}\quad h\in \h_{\bar{0}},\;\; i\in I.\]


Now $\mathbf{U}(\q_{\bar 0})L(\lambda)_{\lambda} \subset \sum_{k_{\lambda}\in L(\lambda)_{\lambda}} \mathbf{U}(\q_{\bar 0})k_{\lambda}\subseteq \bar{L}(\lambda)$ be the $\q_{\bar 0}$-submodule of $\bar{L}(\lambda)$. For any $k_{\lambda}\in L(\lambda)_{\lambda}$, we known that $\mathbf{U}(\q_{\bar 0})k_{\lambda}$ is a highest weight module over $\q_{\bar 0}$ with highest weight $\lambda$ satisfying $f_{\alpha}^{\lambda(h_{\alpha})+1}k_{\lambda} =0$. Thus, $\mathbf{U}(\q_{\bar 0})k_{\lambda}$ is cyclic and $k_{\lambda}$ satisfies \eqref{eq9} for any $k_{\lambda}\in L(\lambda)_{\lambda}$. Then there exists a unique surjective homomorphism of $\q_{\bar 0}$-submodules satisfying 
\[\psi: V(\lambda) \longrightarrow \mathbf{U}(\q_{\bar 0})k_{\lambda}, \quad xu_{\lambda}\mapsto xk_{\lambda}\]
for all $x\in \mathbf{U}(\q_{\bar 0})$. Since $\psi$ is surjective and $V(\lambda)$ is finite dimensional, it follows that  $\mathbf{U}(\q_{\bar 0})k_{\lambda}$  finite dimensional for any $k_{\lambda}\in L(\lambda)_{\lambda}$ and hence $\mathbf{U}(\q_{\bar 0})L(\lambda)_{\lambda}$ is finite dimensional.
\end{proof}



\begin{prop}\label{prop4}
For highest weight $\lambda \in \Lambda^{+}$, consider finite dimensional highest weight $\q$-module $V$. Then there exists a surjective homomorphism of $\q$-modules $\psi_1: \bar{L}(\lambda)\longrightarrow V$ up to $\Pi$. Moreover, there exists a unique upto $\Pi$ submodule $W$ of $\bar{L}(\lambda)$  such that $V \cong \bar{L}(\lambda)/W$ or $V \cong \Pi\left(\bar{L}(\lambda)\right)/ \Pi\left(W\right)$
 as $\q$-modules.
\end{prop}

\begin{proof}
Consider the highest weight $\q$-module $V$, with highest weight $\lambda$, is generated by an irreducible $\bb_{+}$-module $\mathbf{v}$. For $v_{\lambda}\in \mathbf{v}$ the first two relations in \eqref{eq9} hold. Since $V$ is finite dimensional, then by Proposition \ref{prop4a}, $ f_{i}^{\lambda(h_{i})+1}v_{\lambda}=0$, for all $v_{\lambda}\in {\bf v}, i\in I$. Also, by Remark \ref{remprop12}, the dimension of the highest weight space $\bar{L}(\lambda)_{\lambda}$ is equal to the dimension of the highest weight space ${\bf v}$. Thus the map $\psi_1: \bar{L}(\lambda)\longrightarrow V$ induced by $\bar{L}(\lambda)_{\lambda} \longrightarrow \mathbf{v}$, is a surjective homomorphism of $\q$-modules up to $\Pi$. Since module homomorphism preserve weight spaces,  the kernel of $\psi_1$  is unique upto $\Pi$ and say $\ker(\psi_1)=W$.
\end{proof}


Since every simple finite dimensional $\q$-module is a highest weight module with highest weight $\lambda \in \Lambda^+$, Proposition \ref{prop4} applies to all simple finite dimensional $\q$-modules.

\section{Global Weyl modules}

Let $A$ denote a finitely generated commutative associative unital algebra and $\q(n)=: \q$ with $n\geq 2$. Take $\q\otimes A$, with $\Z_2$-grading is given by $(\q\otimes A)_{j} = \q_j \otimes A, j \in \Z_2$. Then $\q\otimes A$ with  bracket of any two homogeneous elements 
\[[x\otimes a, y\otimes b] := [x, y]\otimes ab,\quad x, y\in \q_{j}, a, b \in A\] is a Lie superalgebra.

Further, we identify $\q$ with a subalgebra of $\q\otimes A$ via the isomorphism $\q\cong \q\otimes \C$ and the inclusion $\q\otimes \C \subseteq \q\otimes A$. Let $\mathcal{I}$ be the full subcategory of the category of $\q$ modules whose objects are those modules that are isomorphic to direct sums of irreducible finite dimensional $\q_{\bar 0}$ modules.  Note that if $V \in \mathcal{I}$ then any element of $V$  lies in a finite dimensional $\q_{\bar{0}}$ submodule of $V$. Let $\mathcal{I}_{\q\otimes A, \q_{\bar 0}}$ denote the full subcategory of the category of $\q\otimes A$-modules whose objects are the $\q\otimes A$-modules whose restriction to $\q_{\bar 0}$ lies in $\mathcal{I}$.
\begin{lemma}\label{lemclosed}
Category $\mathcal{I}$ is closed under taking submodules, quotients, arbitrary direct sums and finite tensor products.
\end{lemma}

Regard $U(\q \otimes A)$ as a right $\q$-module via right multiplication and given a left $\q$-
module $V$ (up to $\Pi$), set
\begin{equation}
P_A(V) := \mathbf{U}(\q\otimes A) \otimes_{\mathbf{U}(\q)} V.
\end{equation}
Then $P_A(V)$ is left $\q \otimes A$-module by left multiplication and we have an isomorphism of vector spaces
\begin{equation}
P_A(V)\cong \mathbf{U}(\q\otimes A_+)\otimes_{\mathbb{C}} V 
\end{equation}
where $A_+$ is a vector space complement to $\C\subseteq A$.  Note that $P_{A}(V)$ is defined up to $\Pi$.
\begin{lemma} 
Let $V$ be a $\q$-module whose restriction to $\q_{\bar 0}$ lies in $\mathcal{I}$. Then $P_A(V) \in \mathcal{I}_{\q\otimes A, \q_{\bar 0}}$. 
\end{lemma}

\begin{proof}
Note that $\q$ is finitely semisimple, that is, it is isomorphic to direct sums of irreducible finite dimensional $\q_{\bar 0}$-modules via adjoint representation. So $\q \in \mathcal{I}$ and $\q\otimes A \cong \q^{\oplus \dim (A)}$ as $\q_{\bar{0}}$-modules.  By Lemma \ref{lemclosed}, $\q\otimes A$ is a finitely semisimple $\q_{\bar 0}$-module, that is, it is isomorphic to direct sums of irreducible finite dimensional $\q_{\bar 0}$ modules. Again, by Lemma \ref{lemclosed}, as $\mathcal{I}$ is closed under finite tensor product and arbitrary direct sum so $\mathbf{U}(\q\otimes A)$ is finitely semisimple $\q_{\bar 0}$-module. Also as $V  \in \mathcal{I}$, we have $\mathbf{U}(\q\otimes A)\otimes_{\mathbb{C}} V$ is a finitely semisimple   $\q_{\bar 0}$-module. Consider the map 
\begin{equation}\label{eq3.2}
\mathbf{U}(\q\otimes A)\otimes_{\mathbb{C}} V \to \mathbf{U}(\q\otimes A)\otimes_\mathbf{U(\q)} V, \quad u\otimes v \mapsto u\otimes v.
\end{equation} 
For every $u\in \mathbf{U}(\q\otimes A), \;v\in V$ and every homogenenous element $x\in \q$, we have 
\[x\cdot (u\otimes v) = xu \otimes v = ([x, u] + (-1)^{|x| |u|} u x)\otimes v = [x, u]\otimes v + (-1)^{ |x|  |u|} u\otimes x\cdot v.\] 
Hence, the map in \eqref{eq3.2} is a surjective homomorphism of $\q$-modules. This shows that $P_A(V)$ is a quotient of $\mathbf{U}(\q\otimes A)\otimes_{\mathbb{C}} V$. Hence the lemma follows from Lemma \ref{lemclosed}.
\end{proof}

\begin{prop}
If $\lambda \in \Lambda^+$, then $P_A(\bar{L}(\lambda))$ (up to $\Pi$) is generated as a left $\mathbf{U}(\q\otimes A)$-module by $L(\lambda)_{\lambda}$ satisfying the following relations:
\begin{align}\label{eq14}
\begin{split}
\n^+p_{\lambda} = 0,\;\;\; &hp_{\lambda} = \lambda(h) p_{\lambda},\;\quad f_{i}^{\lambda(h_{i})+1}p_{\lambda}  =0, \quad (p_{\lambda} :=1\otimes k_{\lambda})\\
& \;\; \mbox{for all}\quad k_{\lambda}\in  L(\lambda)_{\lambda},  h\in \h_{\bar{0}},\;\; \alpha_{i} \in \Delta(\q_{\bar{0}}), i \in I.
\end{split}
\end{align}
\end{prop}
\begin{proof}
Note that $p_{\lambda} = 1\otimes k_{\lambda} \in P_A(\bar{L}(\lambda))$ when $k_{\lambda}\in \bar{L}(\lambda)$. Since $k_{\lambda}$ satisfies the relation in \eqref{eq9}, $1\otimes k_{\lambda}$ satisfies relations \eqref{eq14}. We have to check these are all the relations. To do this, suppose that $M$ is the highest weight $\q\otimes A$-module with highest weight $\lambda$, generated by an irreducible $\mathfrak{b}_{+}$-module $\mathfrak{m}$ such that 
\begin{align}\label{eq14a}
\begin{split}
\n^+ m = 0,\;\;\; &hm = \lambda(h) m,\;\quad f_{i}^{\lambda(h_{i})+1}m  =0, \;\; \mbox{for all}\quad m\in  \mathfrak{m},  h\in \h_{\bar{0}},\;\; i\in I.
\end{split}
\end{align}
By Remark \ref{remprop12}, the dimension of the highest weight space $P_A(\bar{L}(\lambda)_{\lambda})$ is equal to the dimension of the highest weight space $\mathfrak{m}$. Then we have a surjective homomorphism (up to $\Pi$) of $\q\otimes A$-modules $\phi: M \longrightarrow P_A(\bar{L}(\lambda))$  induced by $\phi(\mathfrak{m})  = P_A(\bar{L}(\lambda)_{\lambda})$.

From \eqref{eq14}, let $m \in\mathfrak{m}$ generates $\q$-submodule $M'$ of $M$ which is isomorphic to $\bar{L}(\lambda)$. Thus, the map  $\psi: P_A(\bar{L}(\lambda)) \longrightarrow M$ induced by $\bar{L}(\lambda)\to M'$ is a surjective homomorphism (up to $\Pi$). Since $\phi =\psi^{-1}$, we have $M \cong  P_A(\bar{L}(\lambda))$ or $M \cong  \Pi\left(P_A(\bar{L}(\lambda))\right)$.
\end{proof}
For $\nu \in \Lambda^+$ and $M \in \mathcal{I}_{\q\otimes A, \q_{\bar 0}}$, let $M^{\nu}$ be the unique maximal $\q\otimes A$-module quotient of M satisfying
\[\mbox{wt}(M^{\nu})\subset \nu - Q^+,\] or equivalently, 

\begin{equation}\label{eq15}
M^{\nu} := M/\sum_{\mu \not \in \nu -Q^{+}} \mathbf{U}(\q\otimes A) M_{\mu}.
\end{equation}
 Let $\mathcal{I}_{\q\otimes A, \q_{\bar 0}}^{\nu}$ be the full subcategory of $\mathcal{I}_{\q\otimes A, \q_{\bar 0}}$ whose objects are the left $\mathbf{U}(\q\otimes A)$-modules $M \in \mathcal{I}_{\q\otimes A, \q_{\bar 0}}$ such that $M^{\nu} = M$.

\begin{defn}[Global Weyl module]
Let $\lambda \in \Lambda^+$. We define the global Weyl module (up to $\Pi$)  associated to $\lambda  \in \Lambda^+$ to be 
\[W_{A}(\lambda):= P_A(\bar{L}(\lambda))^{\lambda}.\]
\end{defn}
From
\begin{align*}
&\bar{L}(\lambda)\longrightarrow P_A(\bar{L}(\lambda)) \longrightarrow P_A(\bar{L}(\lambda))/\sum_{\mu \not \in \lambda -Q^{+}} \mathbf{U}(\q\otimes A) M_{\mu},
\end{align*} we note that  $w_{\lambda}$ is the image of $k_{\lambda}$ in $W_{A}(\lambda)$.

The next result gives a description of global Weyl modules by generators and relations. 
\begin{prop}\label{prop3.5}
For $\lambda \in \Lambda^+$, the global Weyl module $W_A(\lambda)$ (up to $\Pi$)  is generated by $W_{A}(\lambda)_{\lambda}$ with defining relations 
\begin{equation}\label{eq20}
(\n^+ \otimes A)w_{\lambda} = 0, \quad hw_{\lambda} = \lambda(h) w_{\lambda},\quad f_{i}^{\lambda(h_{i})+1}w_{\lambda} =0, \quad \mbox{for all}\;\; h\in \h_{\bar{0}},\;\; \alpha_{i}\in \Delta(\q_{\bar{0}})
\end{equation}
and $w_{\lambda}$  in $W_{A}(\lambda)_{\lambda}$. 
\end{prop}
\begin{proof}
Let for any $k_{\lambda}\in \bar{L}(\lambda)_{\lambda}$, $w_{\lambda}$ is the image of $k_{\lambda}$ in $W_{A}(\lambda)$. Note that $(\q_{\alpha}\otimes A)V_{\mu}\subseteq V_{\mu+ \alpha}$ for all $\alpha \in \Delta, \mu\in \h_{\bar{0}}^{*}$. Since the weights of  $W_A(\lambda)$ lie in $\lambda - Q^{+}$, it follows that $(\n^+ \otimes A)w_{\lambda} = 0$. The remaining relations are satisfied by $w_{\lambda}$ since they are satisfied by $k_{\lambda}$. To prove that these are the only relations, let $W'(\lambda)$ be the highest weight module generated by an irreducible $\mathfrak{b}_{+}$-module $\mathfrak{m}$ with relations \begin{equation}\label{eq16a}
(\n^+ \otimes A)m = 0, \quad hm = \lambda(h) m,\quad f_{i}^{\lambda(h_{i})+1}m =0, \quad \mbox{for all}\;\; m \in \mathfrak{m}, h\in \h_{\bar{0}},\;\; i\in I.
\end{equation}
Then we have a surjective homomorphism $\phi: W'(\lambda) \to W_{A}(\lambda)$ induced by $\phi(\mathfrak{m}) = W_{A}(\lambda)_{\lambda}$. Note that the relations \eqref{eq16a} implies $W'(\lambda)$ is a weight module. So  $m \in\mathfrak{m}$  generates $\q$-submodule $W''$ of $W'(\lambda)$ which is isomorphic to $\bar{L}(\lambda)$. Thus, the map  $\psi: P_A(\bar{L}(\lambda)) \longrightarrow W'$ induced by $\bar{L}(\lambda)\to W''$ is a surjective homomorphism. Further, $\q$-weights of $W'(\lambda)$ are bounded above by $\lambda$, it follows that $\psi$ induces a map $W_{A}(\lambda) \to W'(\lambda)$ inverse to $\phi$.
\end{proof}

\begin{thm}\label{thm3.6}
Any global Weyl module with highest weight $\lambda$ in the category $I(\q\otimes A, \q_{\bar{0}})$  is isomorphic to $W_A(\lambda)$ or $\Pi(W_A(\lambda))$. Furthermore, if any object $V(\lambda)\in I(\q\otimes A, \q_{\bar{0}})$ is generated by an irreducible $\mathfrak{b}_{+}$-module $\mathbf{v}$ of weight $\lambda$, then there exists a surjective homomorphism form $W_A(\lambda)$ to $V(\lambda)$ (up to $\Pi$). 
\end{thm}

\begin{proof}
Let $V(\lambda)\in I(\q\otimes A, \q_{\bar{0}})$ be highest weight $\q\otimes A$-module with highest weight $\lambda$ is generated by an irreducible $\mathfrak{b}_{}+$-module $\mathbf{v}$. Then by definition 
\[(\n^+ \otimes A)v= 0, \quad hv= \lambda(h) v, \quad \mbox{for all}\;\; h\in \h_{\bar{0}},\;\; v\in \mathbf{v}.\]
Since the $\q_{\bar{0}}$-module generated by $\mathbf{v}$ is finite dimensional, then by Proposition \ref{prop4a}, $f_{i}^{\lambda(h_{i})+1}v =0$ for all $v\in \mathbf{v}$ and $ \alpha_{i} \in \Delta(\q_{\bar{0}}), i\in I$. Thus, by Proposition \ref{prop3.5}, we have a surjective homomorphism $W_{A}(\lambda)\to V(\lambda)$ induced by $W_A(\lambda)_{\lambda} \to \mathbf{v}$ (up to $\Pi$).

Suppose $W_A'(\lambda)$ is another object in $I(\q\otimes A, \q_{\bar{0}})$ that is generated  by an irreducible $\mathfrak{b}_{+}$-module $\mathbf{m}$ with highest weight $\lambda$ and admits a surjective homomorphism to any object of  $I(\q\otimes A, \q_{\bar{0}})$ which is also generated by  highest weight vectors of weight $\lambda$. In particular, we have a surjective homomorphism $\phi: W_A'(\lambda)\to W_A(\lambda)$. It follows from PBW theorem that $W_A(\lambda)_{\lambda} = \mathbf{U}(\h \otimes A_{+})\otimes_{\C}\mathbf{m}$. Hence the elements of this weight space that generate $W_A(\lambda)$ are the $\C$-multiples of $m$ for all $m\in\mathbf{m}$. Thus, we have $\phi(\mathbf{m}) = W_A(\lambda)_{\lambda}$. Now the relation \eqref{eq20} hold for all $m\in \mathbf{m}$. Thus, there exists a homomorphism $\psi: W_A(\lambda)\to W_A'(\lambda)$ induced by $W_A(\lambda)_{\lambda} \to \mathbf{m}$ (up to $\Pi$) and hence $W_A(\lambda)\cong W_A'(\lambda)$ or $\Pi(W_A(\lambda)) \cong W_A'(\lambda)$.
\end{proof}

\section{Local Weyl modules}
An ideal $I$ of $A$ is said to be of finite co-dimension if $\dim A/I$ is finite. Let
 \[\mathcal{L}(\h\otimes A) = \{\psi\in (\h_{\bar 0} \otimes A)^{*}\mid \psi(\h_{\bar 0} \otimes I) = 0, \mbox{for some finite co-dimensional ideal}\; I \subseteq A\}.\]
For any $\psi \in \mathcal{L}(\h\otimes A)$, there exists  unique, up to $\Pi$, simple finite dimensional $\h\otimes A$-module $H(\psi)$ such that $xv = \psi(x)v$, for all $x\in \h_{\bar{0}}\otimes A$ and $v\in H(\psi)$ (see \cite[Th. 4.3]{CMS16}). Define an action of $\bb \otimes A$ on $H(\psi)$ by $\n^+\otimes A$ to act by zero. Then consider the induced module 
\[\bar{V}(\psi):=\mathbf{U}(\q\otimes A) \otimes_{\mathbf{U}(\bb \otimes A)}H(\psi),\]
which is a highest weight module. Notice that $\bar{V}(\psi)$ is defined up to the parity reversing functor $\Pi$. Further, a submodule of $\bar{V}(\psi)$ is proper if and only if its intersection with $H(\psi)$ is zero. Moreover, any $\q\otimes A$-submodule of a weight module is also a weight module. Hence, if $W\subset \bar{V}(\psi)$ is proper $\q\otimes A$-submodule, then 
\[W= \bigoplus_{\mu\neq \lambda}W_{\mu}, \quad \mbox{where}\;\;\lambda= \psi \mid_{\h_{\bar{0}}}.\]
Therefore, $\bar{V}(\psi)$ has a unique maximal proper submodule $N(\psi)$ 
    \[V(\psi) = \bar{V}(\psi)/N(\psi)\] is an irreducible highest weight $\q\otimes A$-module. So, every finite dimensional irreducible $\q\otimes A$-module is isomorphic to $V(\psi)$ for some $\psi\in \mathcal{L}(\h\otimes A)$(see \cite[Proposition 5.4]{CMS16}). Note that the highest weight space of $V(\psi) \cong_{\h\otimes A} H(\psi)$ or  $V(\psi) \cong_{\h\otimes A} \Pi(H(\psi))$.
    
\begin{defn}\label{def4.1}
Let $ \psi \in  \mathcal{L}(\h\otimes A)$ such that $\lambda=\psi \mid_{\h_{\bar{0}}} \in \Lambda^{+}$. We define the {\em local Weyl module} $W_A^{\mbox{loc}}(\psi)$ associated to $\psi$ upto $\Pi$ to be the $\q \otimes A$-module generated by $H(\psi)$ with defining relations

\[(\n^{+} \otimes A)w_{\psi} = 0, \quad xw_{\psi} = \psi(x)w_{\psi},\quad f_{i}^{\lambda(h_{i})+1}w_{\psi} =0,  \quad \mbox{for all} \quad w_{\psi} \in H({\psi}), x\in \h_{\bar{0}}\otimes A,\;\; i\in I.\]

\end{defn}

A $\q \otimes A$-module generated by  $H({\psi})$ is called {\em highest map-weight module} with {\em highest map-weight} $\psi$ if 
\[(\n^{+} \otimes A)w_{\psi} = 0, \quad xw_{\psi} = \psi(x)w_{\psi}, \quad \mbox{for all}\quad w_{\psi} \in H({\psi}), x\in \h_{\bar{0}}\otimes A.\] A vector $w_{\psi} \in H(\psi)$ is called {\em highest map-weight vector} of highest map-weight $\psi$.

Recall the even part of queer Lie superalgebra $\q(=\q_{\bar{0}}\oplus \q_{\bar{1}})$ is isomorphic to $\mathfrak{gl}(n+1)$.  For each $\alpha \in \Phi_{\bar{0}}^{+}$ we have an $\mathfrak{sl}(2)$-triple $x_{\alpha}, y_{\alpha}, h_{\alpha}$.
\begin{lemma}
Suppose $\psi \in \mathcal{L}(\h\otimes A)$ such that $\lambda= \psi \mid_{\h_{\bar{0}}} \in \Lambda^{+}$. Consider $\q \otimes A$-module generated by an irreducible module $H({\psi})$. If $\alpha_i \in \Phi_{\bar{0}}^{+},$ then $f_{i}^{\lambda(h_{i})+1}w_{\psi}=0$ for all $w_{\psi} \in H({\psi}), i\in I$.
\end{lemma}
\begin{proof}
Let $\h= \h_{\bar{0}} \oplus \h_{\bar{1}}$ be the Cartan subalgebra of $\q$ and $\h \otimes A$ be the Cartan subalgebra of $\q \otimes A$. Since $\h \subset \h \otimes A$ is a subalgebra, $h w_{\psi}=\lambda(h)w_{\psi}$ for all $h \in \h_{\bar{0}}, w_{\psi} \in H({\psi})$ and $\h_{\bar{1}}w_{\psi}=0$. So $\lambda$ is the highest weight with highest weight vector $w_{\psi}$. The vector $f_{i}^{\lambda(h_{i})+1}w_{\psi}$ has weight $\lambda-(\lambda(h_{i})+1)\alpha_i$. On the other hand, by Theorem  \ref{thm3.6}, $W_A^{\mbox{loc}}(\psi)$ is a quotient of the global Weyl module $W_A(\lambda)$ up to $\Pi$, hence is direct sum of finite-dimensional irreducible $\q_{\bar{0}}$-modules. This implies the weights of $W_A^{\mbox{loc}}(\psi)$ are invariant under the action of Weyl group of $\q_{\bar{0}}$.  Let $s_{\alpha_i}$ denotes the reflection associated to the root $\alpha_i$. Then $s_{\alpha_i}(\lambda-(\lambda(h_{i})+1)=\lambda+\alpha_i$, and this implies $f_{i}^{\lambda(h_{i})+1}w_{\psi}=0$.
\end{proof}

Let $u$ be an indeterminate and for $a\in A, \alpha\in  \Phi_{\bar{0}}^{+}$, define a power series with coefficients in $\mathbf{U}(\h\otimes A)$ by 
\[\mathbf{p}_{a, \alpha}(u) = \exp \left(-\sum_{r=1}^{\infty} \frac{h_{\alpha}\otimes a^r}{r}u^r\right).\]
For $i\in \N$, let $p_{a, \alpha}^i$ be the coefficient of $u^i$ in $\mathbf{p}_{a, \alpha}(u)$.
\begin{lemma}\label{lem4.4}
Suppose $r \in \mathbb{N}$, $a \in A$ and $\alpha \in \Phi_{\bar{0}}^{+}$ then 
\begin{equation}\label{eq4.1}
(x_{\alpha} \otimes a)^{r}(y_{\alpha}\otimes 1)^{r+1}-(-1)^{r} \sum_{i=0}^{r} (y_{\alpha}\otimes a^{r-i})p_{a, \alpha}^{i} \in \mathbf{U}(\g \otimes A)(\n^{+}\otimes A).
\end{equation}
\end{lemma}
\begin{proof}
When $A = \C[t^{\pm 1}]$, the formula in \eqref{eq4.1} is proved in \cite{CP01}. Further since the fact that $t$ is an invertible element in $\C[t^{\pm 1}]$ is not used in that proof, the result is still true when $A = \C[t]$. Applying, the Lie algebra homomorphism,
\[\mathfrak{sl}(2)\otimes \C[t]\to \mathfrak{sl}(2)\otimes A, \quad x\otimes t^r \mapsto x\otimes a^r, \quad r\in \N, \;\;x\in \mathfrak{sl}(2)\] gives the result.
\end{proof}
From now on we assume that $A$ is finitely generated (say $a_1, \ldots, a_m$ be a set of generators of $A$). Using the first and third relation of the Definition \ref{def4.1} and Lemma \ref{lem4.4}, and then applying induction, one can prove the following. 
\begin{lemma}\label{lem4.5}
Suppose $\psi \in \mathcal{L}(\h\otimes A)$ such that $\lambda= \psi \mid_{\h_{\bar{0}}} \in \Lambda^{+}$. If $\alpha \in \Phi_{\bar{0}}^+$, $a_{1}, a_{2}, \dots, a_{m} \in A$, and  $s_{1}, s_{2}, \cdots, s_{m} \in \mathbb{N}, w_{\psi}\in H(\psi)$, then
\begin{equation}
(y_{\alpha}\otimes a_{1}^{s_{1}} \cdots a_{t}^{s_{m}})w_{\psi} \in \mathrm{span}_{\mathbb{C}}\{(y_{\alpha}\otimes a_{1}^{\ell_{1}} \cdots a_{m}^{\ell_{m}})w_{\psi} \mid 0 \leq \ell_{i}<\lambda({h_{\alpha}}), i=1, \cdots, m\}. 
\end{equation}
In particular, $(y_{\alpha}\otimes A)H\psi)$ is finite dimensional.
\end{lemma}

\begin{lemma}\label{prop4.6}
If $\psi \not \in \mathcal{L}(\h\otimes A)$, then $W_A^{\mbox{loc}}(\psi) = 0$.
\end{lemma}

\begin{proof}
Let  $\lambda= \psi \mid_{\h_{\bar{0}}} \in \Lambda^{+}$, let $\alpha$ be a positive root of $\q$ and let $I_{\alpha}$ be the kernel of the linear map 
\begin{align*}
A&\to \Hom_{\C}(W_A^{\mbox{loc}}(\psi)_{\lambda}\otimes \q_{-\alpha}, (\q_{-\alpha}\otimes A)H(\psi),\\
a&\to (v\otimes u \mapsto (u\otimes a)v), \quad a\in A, v\in W_A^{\mbox{loc}}(\psi)_{\lambda}, \;\; u\in \q_{-\alpha}.
\end{align*}
Notice that $I_{\alpha} = \{a\in A\mid (u\otimes a)v =0, v\in W_A^{\mbox{loc}}(\psi)_{\lambda}, \;\; u\in \q_{-\alpha}\}$. Since $\q_{-\alpha} = \C y_{\alpha}$, Lemma \ref{lem4.5} gives that $(\q_{-\alpha}\otimes A)w_{\psi})$ is finite dimensional. Thus, $I_{\alpha}$ is a linear subspace of $A$ of finite co-dimension. We claim that $I_{\alpha}$ is an ideal of $A$. Since $\alpha\neq 0$, we can choose $h \in \h_{\bar{0}}$ such that $\alpha(h)\neq 0$. Then, for all $b\in A, a\in I_{\alpha}, v\in W_A^{\mbox{loc}}(\psi)_{\lambda}, \;\; u\in \q_{-\alpha}$, we have 
\begin{align*}
0 &= (h\otimes b)(u\otimes a)v\\
&= [h\otimes b, u\otimes a]v+ (u \otimes a)(h\otimes b)v\\
&= -\alpha(h)(u\otimes ba)v+ (u\otimes a)(h\otimes b)v.
\end{align*} 
Since $(h\otimes b)v\in W_A^{\mbox{loc}}(\psi)_{\lambda}$ and $a\in I_{\alpha}$, the last term above is zero. Since we have assumed that $\alpha(h)\neq 0$, this implies that $(u\otimes ba)v$. As this holds for all $v\in W_A^{\mbox{loc}}(\psi)_{\lambda}$ and $u\in \q_{-\alpha}$, we have $ba \in I_{\alpha}$. Hence $I_{\alpha}$ is an ideal of $A$.

Let $I = \bigcap_{\alpha \in  \Phi_{\bar{0}}^{+}}I_{\alpha}$. Since $\q$ is finite dimensional and hence a finite number of positive roots, this intersection is finite and thus $I$ is an ideal of $A$ of finite co-dimension. Then we have 
\[(\n_{\bar{0}}^{-} \otimes I)W_A^{\mbox{loc}}(\psi)_{\lambda}=0. \] Since $\lambda$ is the highest weight of $W_A^{\mbox{loc}}(\psi)_{\lambda}$, we also have $(\n^{+} \otimes A)W_A^{\mbox{loc}}(\psi)_{\lambda}=0$. Further, since $\h_{\bar{0}}\otimes I\subseteq [\n^{+}\otimes A, \n_{\bar{0}}^{-}\otimes I]$, we have $(\h_{\bar{0}} \otimes I)W_A^{\mbox{loc}}(\psi)_{\lambda}=0$. In particular, $(\h_{\bar{0}} \otimes a)w_{\psi} =0$ for all $a\in I$. Since $\psi \not \in \mathcal{L}(\h\otimes A)$, then there exists $a\in I$ such that $\psi(h\otimes a) \neq 0$. So, we must have $w_{\psi} =0$ and hence \[W_A^{\mbox{loc}}(\psi)= \mathbf{U}(\n^{-}\otimes A)w_{\psi} =0.\]
\end{proof}

\begin{defn}
Suppose $\psi \in \mathcal{L}(\h\otimes A)$ such that $\lambda= \psi \mid_{\h_{\bar{0}}} \in \Lambda^{+}$. Let $I_{\psi}$ be the sum of all ideals $I\subseteq A$ such that $(\h_{\bar{0}} \otimes I)W_A^{\mbox{loc}}(\psi)_{\lambda}=0$.
\end{defn}
\begin{rmk}
It follows from the proof of Lemma \label{prop4.6} that, $I_{\psi}$ is a finite co-dimensional ideal in $A$ and that $(y_{\alpha}\otimes I_{\psi}) w_{\psi} = 0$ for all $\alpha \in \Phi_{\bar{0}}^{+}$. Furthermore, since $I_{\psi}$ has finite co-dimension and $A$ is assumed to be finitely generated, we have that $I_{\psi}^n$ has finite co-dimension, for all $n\in \N$ (see \cite[Lemma 2.1(a), (b)]{CLS19})
\end{rmk}

\begin{lemma}\label{lem4.9}
Suppose $\psi \in \mathcal{L}(\h\otimes A)$ such that $\lambda= \psi \mid_{\h_{\bar{0}}} \in \Lambda^{+}$. Then there exists $n_{\psi}\in \N$ such that 
\[(\n^{-}\otimes I_{\psi}^{n_{\psi}}) w_{\psi}= 0\quad \mbox{for all}\;\; w_{\psi}\in H(\psi).\]
\end{lemma}
\begin{proof}
For $\alpha = \sum_{i =1}^n a_i\alpha_i$, with $a_i\in \N$ and where the $\alpha_i$ are the simple roots of $\q$, we define the height of $\alpha$ to be 
\[\mbox{ht}(\alpha) = \sum_{i=1}^n a_i.\] By induction on the height of $\alpha$, we will show that $(\q_{-\alpha}\otimes I_{\psi}^{\mbox{ht}(\alpha)})w_{\psi} = 0$ for all $\alpha\in \Phi^+, w_{\psi}\in H(\psi)$. Since $\q$ is finite dimensional, the heights of elements of $\Phi^+$ are bounded above, and thus the lemma will follow.

For the base case, first we will show that 
\begin{equation}\label{eq4.3}
(f_i\otimes I_{\psi})w_{\psi} = 0, \quad \mbox{for all}\;\; w_{\psi}\in H(\psi), \alpha_i\in \textstyle\sum,
\end{equation}
since the set $\{f_{i}\mid \alpha_i\in \textstyle\sum\}$ of generators of $\n^-$. By the above remark, it suffices to consider the case $\alpha_i\in \textstyle\sum_{\bar{1}}$. At first fix such an $\alpha_i$. By \eqref{eqsystem30}, there exists $\alpha_j\in \Phi_{\bar{1}}$ such that $\alpha_k: = \alpha_i+\alpha_j\in \Phi_{\bar{0}}^{+}$. Note that $\dim \q_{\alpha} = (1\mid 1)$ for any $\alpha\in \Phi$, that is, $\q_{\alpha}$ is generated by an even vector and an odd vector. Further, since $[h, [e_{j}, f_{k}] ] = \alpha_i(h)[e_{j}, f_{k}]$, we can write after rescaling, if necessary
\begin{equation}
[e_{j}, f_{k}] = f_{i}.
\end{equation}
Then 
\begin{align*}
(f_i\otimes I_{\psi})w_{\psi} &= [e_j \otimes A, f_k\otimes I_{\psi}]w_{\psi}\\
&\subseteq (e_j \otimes A)(f_k\otimes I_{\psi}) w_{\psi}+ (f_k\otimes I_{\psi})(e_j \otimes A) w_{\psi}.
\end{align*}
Since by Definition \ref{def4.1},  $(e_j \otimes A) w_{\psi}$ and by the above remark, $(f_k\otimes I_{\psi}) w_{\psi}=0$, then from above, \eqref{eq4.3} holds.

Now suppose that $\alpha_j\in \Phi^+$ with $\mbox{ht}(\alpha_j) >1$. Then there exists $\alpha_k, \alpha_{\ell}\in \Phi^+$ with $\mbox{ht}(\alpha_k), \mbox{ht}(\alpha_{\ell}) < \mbox{ht}(\alpha_i)$ such that $f_i\in \C[f_k, f_{\ell}]$. Then by induction hypothesis
\begin{align*}
(f_i\otimes I_{\psi}^{\mbox{ht}(\alpha)})w_{\psi} &= [f_k\otimes I_{\psi}^{\mbox{ht}(\alpha_k)}, f_{\ell}\otimes I_{\psi}^{\mbox{ht}(\alpha_{\ell})}]w_{\psi}\\
& = (f_k\otimes I_{\psi}^{\mbox{ht}(\alpha_k)})( f_{\ell}\otimes I_{\psi}^{\mbox{ht}(\alpha_{\ell})})(w_{\psi})- ( f_{\ell}\otimes I_{\psi}^{\mbox{ht}(\alpha_{\ell})})(f_k\otimes I_{\psi}^{\mbox{ht}(\alpha_k)})(w_{\psi})\\
&=0.
\end{align*}

\end{proof}

\begin{cor}\label{cor4.9}
Suppose $\psi \in \mathcal{L}(\h\otimes A)$ such that $\lambda= \psi \mid_{\h_{\bar{0}}} \in \Lambda^{+}$ and let $n\in \N$ as in Lemma \ref{lem4.9}. Then  
\[(\q\otimes I_{\psi}^{n_{\psi}}) H(\psi)= 0.\]
\end{cor}
\begin{proof}
By \cite[Lemma 2.12]{CMS16}, to prove that $(\q\otimes I_{\psi}^{n_{\psi}}) H(\psi)=0$, it is enough to prove that $(\q\otimes I_{\psi}^{n_{\psi}}) w_{\psi}=0$ for any non-zero $w_{\psi}\in H(\psi)$. Since the triangular decomposition of $\q= \n^{-}+\h+\n^+$, we have to show that
\[(\n^{-}\otimes I_{\psi}^{n_{\psi}}) w_{\psi}= 0, \quad (\h \otimes I)^{n_{\psi}} w_{\psi} =0, \quad \n^{+}\otimes I_{\psi}^{n_{\psi}}) w_{\psi}= 0.\]
From the first relation in Definition \ref{def4.1} we have $(\n^{+}\otimes I_{\psi}) w_{\psi}= 0$ for all $w_{\psi}\in H(\psi)$ and this implies $(\n^{+}\otimes I_{\psi}^{n_{\psi}}) w_{\psi}= 0$. Also, $(\h_{\bar{0}} \otimes I) w_{\psi} =0$ by the definition of $I_{\psi}$. Again by \cite[Lemma 4.1]{CMS16}, $(\h_{\bar{1}} \otimes I) w_{\psi} =0$. From Lemma \ref{lem4.9}, we get $(\n^{-}\otimes I_{\psi}^{n_{\psi}}) w_{\psi}= 0$.
\end{proof}
Now we give a sufficient conditions for local Weyl modules to be finite dimensional.

\begin{thm}
Assume that $A$ is finitely generated. Then the local Weyl module $W_A^{\mbox{loc}}(\psi)$ is finite dimensional for all $\psi \in \mathcal{L}(\h\otimes A)$ such that $\lambda= \psi \mid_{\h_{\bar{0}}} \in \Lambda^{+}$.
\end{thm}

\begin{proof}
By Definition \ref{def4.1}, we have $W_A^{\mbox{loc}}(\psi) = \mathbf{U}(\n^-\otimes A)H(\psi)$. Also, by Lemma \ref{lem4.9}, we have $(\n^{-}\otimes I_{\psi}^{n_{\psi}}) H(\psi)= 0$. Thus, 
\[W_A^{\mbox{loc}}(\psi) = \mathbf{U}(\n^-\otimes A/I_{\psi}^{n_{\psi}})H(\psi).\] Since the set of $\q$-weights of $W_A^{\mbox{loc}}(\psi)$ is finite, there exists $N\in \N$ such that 
\[W_A^{\mbox{loc}}(\psi) = \mathbf{U}_{n}(\n^-\otimes A/I_{\psi}^{n_{\psi}})H(\psi), \;\;\mbox{for all}\,\;\; n\geq N,\] where $ \mathbf{U}(\mathfrak{g}) =  \sum_{n=0}^{\infty}\mathbf{U}_{n}(\g)$ is the usual filtration on the universal enveloping algebra of a Lie superalgebra $\g$ induced from the natural grading on the tensor algebra. Since the Lie superalgebra $\n^-\otimes A/I_{\psi}^{n_{\psi}}$ and $H(\psi)$ are finite dimensional, the local Weyl module $W_A^{\mbox{loc}}(\psi)$ is finite dimensional. 
\end{proof}

\begin{thm}\label{thm4.11}
Let $\psi \in \mathcal{L}(\h\otimes A)$ such that $\lambda= \psi \mid_{\h_{\bar{0}}} \in \Lambda^{+}$. Then any finite dimensional local Weyl module with highest map-weight $\psi$ in the category $I(\q\otimes A, \q_{\bar{0}})$  is isomorphic to $W_A^{\mbox{loc}}(\psi)$ or $\Pi W_A^{\mbox{loc}}(\psi)$. Furthermore, if any finite dimensional object $V(\psi)\in I(\q\otimes A, \q_{\bar{0}})$ is generated by an irreducible module $H(\psi)$ of map-weight $\psi$, then there exists a surjective homomorphism form $W_A^{\mbox{loc}}(\psi)$ to $V(\psi)$ (up to $\Pi$). 
\end{thm}

\begin{proof}
Let $V(\psi)\in I(\q\otimes A, \q_{\bar{0}})$ be a finite dimensional highest map-weight $\q\otimes A$-module with highest map-weight $\psi$ is generated by an irreducible module $H(\psi)$. Then by Definition \ref{def4.1},
\[(\n^+ \otimes A)v= 0, \quad hv= \lambda(h) v, \quad \mbox{for all}\;\; h\in \h_{\bar{0}},\;\; v\in H(\psi).\]
Since the $\q_{\bar{0}}$-module generated by $H(\psi)$ is finite dimensional, then $f_{i}^{\lambda(h_{i})+1}v =0$ for all $v\in H(\psi)$ and $i\in I$. Thus, we have a surjective homomorphism $W_{A}^{\mbox{loc}}(\psi)\to V(\psi)$ induced by $W_A^{\mbox{loc}}(\psi)_{\lambda} \to H(\psi)$ (up to $\Pi$).

Suppose $W_A'(\psi)$ is another object in $I(\q\otimes A, \q_{\bar{0}})$ that is generated  by an irreducible module $\mathbf{m}(\psi)$ with highest map-weight $\psi$ and admits a surjective homomorphism to any object of  $I(\q\otimes A, \q_{\bar{0}})$ which is also generated by  highest map-weight vectors of map-weight $\psi$. Then $W_A'(\psi)$ is a quotient of $W_A^{\mbox{loc}}(\psi)$ (up to $\Pi$) and vice-versa. Since both modules are finite dimensional,  $W_A^{\mbox{loc}}(\psi)\cong W_A'(\psi)$ or $\Pi(W_A^{\mbox{loc}}(\psi)) \cong W_A'(\psi)$.
\end{proof}

\begin{cor}
Let $\psi \in \mathcal{L}(\h\otimes A)$ such that $\lambda= \psi \mid_{\h_{\bar{0}}} \in \Lambda^{+}$. Then the local Weyl module $W_A^{\mbox{loc}}(\psi)$ is the maximal finite dimensional quotient of the global Weyl module $W_A(\lambda)$ (up to $\Pi$) that is a highest map-weight module of highest map-weight $\psi$.
\end{cor}

\begin{rmk}
By \cite[Theorem 5.6]{CMS16}, any finite dimensional $\q\otimes A$-module is a highest map-weight module for some $\psi \in \mathcal{L}(\h\otimes A)$ such that $\lambda= \psi \mid_{\h_{\bar{0}}} \in \Lambda^{+}$. Then by Theorem \ref{thm4.11}, there exists a surjective homomorphism from the local Weyl module $W_A^{\mbox{loc}}(\psi)$ to such a module (up to $\Pi$). Equivalently, all finite dimensional $\q\otimes A$-modules are quotients of loal Weyl modules (up to $\Pi$).
\end{rmk}

\section{ Tensor product of local Weyl modules}

If $A$ and $B$ are associative unitary algebras, then all irreducible representations of $A\otimes B$ are of the form $V_A\otimes V_{B}$. Further, all such modules are irreducible. However, when $A$ and $B$ are allowed to be superalgebras, then  $V_A\otimes V_{B}$ is not necessarily irreducible.

If $\g_i$ for $i=1, 2$ are two finite dimensional Lie superalgebras, and $V_i$ is an irreducible finite-dimensional $\g_i$-module for $i=1, 2$, then $\g_1\oplus \g_2$-module $V_1\otimes V_2$ is irreducible only if $\End_{\g_i}(V_i)_{\bar{1}} = 0$ for some $i =1, 2$. When $\End_{\g_i}(V_i)_{\bar{1}} = \C\phi_i, \phi_i^2 = -1$ for $i =1$ and $i=2$, we have
\[\widehat{V} = \{v\in V_1\otimes V_2 \mid (\tilde{\phi_1}\otimes \phi_2)(v) = v\}, \quad \mbox{where}\;\; \tilde{\phi_1} = \sqrt{-1}\phi_1\] is an irreducible $\g_1\oplus \g_2$-submodule of $V_1\otimes V_2$ such that $V_1\otimes V_2\cong \widehat{V} \oplus \widehat{V}$ (see \cite[p.27]{CHE95}). Now we set
\[V_1\widehat{\otimes}V_2 = \begin{cases}
V_1\otimes V_2& \mbox{if $V_1\otimes V_2$ is irreducible},\\
\widehat{V}\subsetneq V_1\otimes V_2& \mbox{if $V_1\otimes V_2$ is not irreducible}.
\end{cases}\] 
If $V_i$ is an irreducible finite-dimensional $\g_i$-module for $i=1, 2$, then it is proved that every irreducible finite dimensional $\g_1\oplus \g_2$-module is isomorphic to a module of the form $V_1\widehat{\otimes}V_2$ (see \cite[Prop. 8.4]{CHE95}).

Given an ideal $I$ of $A$, we define its support to be the set 
\[\mbox{Supp}(I) = \{\m \in \mbox{MaxSpec$(A)$} \mid I\subseteq \m \}.\]

\begin{thm}\label{tensorprod}
Assume that $A$ is finitely generated.  Let $\psi_{1}, \psi_2 \in \mathcal{L}(\h\otimes A)$ such that $\lambda_1\mid_{\h_{\bar{0}}}= \psi_1, \lambda_2\mid_{\h_{\bar{0}}}= \psi_2$ and suppose that $\lambda_1, \lambda_2 \in \Lambda^{+}$ such that $\lambda_1+\lambda_2\in \Lambda^{+}$. If $\mbox{Supp}(I_{\psi_1}) \cap \mbox{Supp}(I_{\psi_2}) = \emptyset$, then we have
\[W_A^{\mbox{loc}}(\psi_1)\otimes W_A^{\mbox{loc}}(\psi_2) \cong \begin{cases}
W_A^{\mbox{loc}}(\psi_1 +\psi_2), \; \mbox{or}\\
W_A^{\mbox{loc}}(\psi_1 +\psi_2)\oplus W_A^{\mbox{loc}}(\psi_1 +\psi_2)
\end{cases}\] as $\q\otimes A$-modules.
\end{thm}

\begin{proof}
Let $\psi_{1}, \psi_2 \in \mathcal{L}(\h\otimes A)$ such that $\lambda_1\mid_{\h_{\bar{0}}}= \psi_1, \lambda_2\mid_{\h_{\bar{0}}}= \psi_2$ and suppose that $\lambda_1, \lambda_2 \in \Lambda^{+}$ such that $\lambda_1+\lambda_2\in \Lambda^{+}$. Let the local Weyl module $W_A^{\mbox{loc}}(\psi_i)$ associated to $\psi_i$ (upto $\Pi$)  be the $\q \otimes A$-module generated by $H(\psi_i)$ for $i=1, 2$.  Let $\rho_i$ be the representation corresponding to $W_A^{\mbox{loc}}(\psi_i)$  for $i=1, 2$. By Corollary \ref{cor4.9}, there exists $n_1, n_1\in \N$ such that $(\q\otimes I_{\psi_i}^{n_i})w_{\psi_i} = 0$ for all $w_{\psi_i}\in H(\psi_i)$  for $i=1, 2$. Notice that $W_A^{\mbox{loc}}(\psi_1)\otimes W_A^{\mbox{loc}}(\psi_2)$, as $(\q\otimes A/I_{\psi_1}^{n_1})\oplus (\q\otimes A/I_{\psi_2}^{n_2})$-module generated by $H(\psi_1)\otimes H(\psi_2)$ is either irreducible or is isomorphic to $\widehat{V} \oplus \widehat{V}$ where $\widehat{V}\subsetneq W_A^{\mbox{loc}}(\psi_1)\otimes W_A^{\mbox{loc}}(\psi_2)$ is an irreducible $(\q\otimes A/I_{\psi_1}^{n_1})\oplus (\q\otimes A/I_{\psi_2}^{n_2})$-module.

Now the representation $\rho_1\otimes \rho_2$ factors through the composition 
\begin{equation}\label{eq5.1}
\q\otimes A \xhookrightarrow{\textsf{\tiny{d}}} (\q\otimes A)\oplus (\q\otimes A) \twoheadrightarrow (\q\otimes A/I_{\psi_1}^{n_1})\oplus (\q\otimes A/I_{\psi_2}^{n_2})
\end{equation}
where first map is the diagonal map and the second map is the projection on each summand. By \cite[Lemma 2.1]{CLS19}, we have that $A= I_{\psi_1}^{n_1}+ I_{\psi_2}^{n_2}$ and $I_{\psi_1}^{n_1} \cap I_{\psi_2}^{n_2} = I_{\psi_1}^{n_1}I_{\psi_2}^{n_2}$, since $\mbox{Supp}(I_{\psi_1}) \cap \mbox{Supp}(I_{\psi_2}) = \emptyset$. Therefore, we have the following commutative diagram:
\begin{center}
$\xymatrix{
\q\otimes A \ar@{->>}[d] \ar@{^{(}->}[r]& (\q\otimes A)\oplus (\q\otimes A) \ar@{->>}[d]\\
\q\otimes A/I_{\psi_1}^{n_1}I_{\psi_2}^{n_2}\ar[r]^{}& (\q\otimes A/I_{\psi_1}^{n_1})\oplus (\q\otimes A/I_{\psi_2}^{n_2}) }$
\end{center}
It follows that the composition \eqref{eq5.1} is surjective.  By the surjective of \eqref{eq5.1}, it follows that $W_A^{\mbox{loc}}(\psi_1)\otimes W_A^{\mbox{loc}}(\psi_2)$, as $(\q\otimes A)$-module generated by $H(\psi_1)\otimes H(\psi_2)$ is either irreducible or is isomorphic to $\widehat{V} \oplus \widehat{V}$  where $\widehat{V}\subsetneq W_A^{\mbox{loc}}(\psi_1)\otimes W_A^{\mbox{loc}}(\psi_2)$. Moreover, $\h_{\bar{0}}\otimes A$ acts on $ w_{\psi_1}\otimes w_{\psi_2}$ as follows:
\begin{align*}
 x (w_{\psi_1}\otimes w_{\psi_2})  &= x w_{\psi_1}\otimes w_{\psi_2}  \oplus w_{\psi_1} \otimes xw_{\psi_2} = \psi_1(x)w_{\psi_1}\otimes w_{\psi_2}  \oplus w_{\psi_1}\otimes\psi_2(x)w_{\psi_2}\\
 &= \psi_1(x)(w_{\psi_1}\otimes w_{\psi_2} ) \oplus \psi_2(x)(w_{\psi_1}\otimes w_{\psi_2})\\
&= (\psi_1+\psi_2)(x)(w_{\psi_1}\otimes w_{\psi_2)}, \quad \mbox{for all}\quad w_{\psi_i} \in H({\psi_i}), x\in \h_{\bar{0}}\otimes A, i=1, 2
 \end{align*} and 

\[(\n^{+} \otimes A)(w_{\psi_1}\otimes w_{\psi_2}) = (\n^{+} \otimes A)w_{\psi_1}\otimes w_{\psi_2} \oplus w_{\psi_1}\otimes  (\n^{+} \otimes A)w_{\psi_2}= 0.\]
Thus, $W_A^{\mbox{loc}}(\psi_1)\otimes W_A^{\mbox{loc}}(\psi_2)$ is a finite dimensional highest map-weight module generated by $H(\psi_1)\otimes H(\psi_2)$ of highest map-weight $\psi_1+\psi_2$. Therefore, by Theorem \ref{thm4.11}, $W_A^{\mbox{loc}}(\psi_1)\otimes W_A^{\mbox{loc}}(\psi_2)$ is a quotient of $W_A^{\mbox{loc}}(\psi_1 +\psi_2)$.

By \cite[Theorem 4.3]{CMS16}, there exists a unique (up to $\Pi)$ irreducible finite dimensional $\h\otimes A$-module $H(\psi_1+\psi_2)$ such that $xv = (\psi_1+\psi_2)(x)v, x\in \h_{\bar{0}}\otimes A, v \in H(\psi_1+\psi_2)$. Let $I = I_{\psi_1} \cap I_{\psi_2}= I_{\psi_1}I_{\psi_2}$ and $n = n_{\psi_1+\psi_2}$. Then $I\subseteq I_{\psi_1+\psi_2}$ and hence $(\h_{\bar{0}}\otimes I_{\psi_1+\psi_2}^n)H(\psi_1+\psi_2) =0$. So, the action of $\mathfrak{b}\otimes A$ on $H(\psi_1+\psi_2)$ descends to an action of $\mathfrak{b}\otimes A/I^n$ on $H(\psi_1+\psi_2)$. Now consider the induced module 
\[M(\psi_1+\psi_2) := \mathbf{U}(\q\otimes A/I_{\psi_1+\psi_2}^n)\otimes_{\mathbf{U}(\mathfrak{b}\otimes A/I^n)} H(\psi_1+\psi_2).\]
It follows that $W_A^{\mbox{loc}}(\psi_1 +\psi_2)$ is a quotient of $M(\psi_1+\psi_2)$.

On the other hand, since $\mathfrak{b}\otimes A$ module $H(\psi_1+\psi_2)$ is irreducible, by \cite[Prop 6.3]{CMS16}, we have  $H(\psi_1+\psi_2)$ and $H(\psi_1)\otimes H(\psi_2)$ are isomorphic. Hence,
\begin{align*}
M(\psi_1+\psi_2) &=  \mathbf{U}(\q\otimes A/I_{\psi_1+\psi_2}^n)\otimes_{\mathbf{U}(\mathfrak{b}\otimes A/I_{\psi_1+\psi_2}^n)} H(\psi_1+\psi_2)\\
&\cong  \mathbf{U}(\q\otimes (A/I_{\psi_1}^{n}\oplus A/I_{\psi_2}^n))\otimes_{\mathbf{U}(\mathfrak{b}\otimes (A/I_{\psi}^n\oplus A/I_{\psi_2}^n))} H(\psi_1)\otimes H(\psi_2)\\
&\cong \left(  \mathbf{U}(\q\otimes (A/I_{\psi_1}^{n})\right)\otimes \left(  \mathbf{U}(\q\otimes (A/I_{\psi_1}^{n})\right) \otimes_{\mathbf{U}(\mathfrak{b}\otimes (A/I_{\psi_1}^{n}))\otimes \mathbf{U}(\mathfrak{b}\otimes (A/I_{\psi_1}^{n}))}(H(\psi_1)\otimes H(\psi_2))\\
& \cong  \left( \mathbf{U}(\q\otimes (A/I_{\psi_1}^{n})) \otimes_{\mathbf{U}(\mathfrak{b}\otimes (A/I_{\psi_1}^{n}))} H(\psi_1)\right)\otimes  \left( \mathbf{U}(\q\otimes (A/I_{\psi_2}^{n})) \otimes_{\mathbf{U}(\mathfrak{b}\otimes (A/I_{\psi_2}^{n}))} H(\psi_2)\right)\\
&= M(\psi_1)\otimes M(\psi_2).
\end{align*}
Since $W_A^{\mbox{loc}}(\psi_1 +\psi_2)$ is a quotient of $M(\psi_1+\psi_2)$, so $W_A^{\mbox{loc}}(\psi_1 +\psi_2)$ is a quotient of $M(\psi_1)\otimes M(\psi_2)$ and hence we can fix a surjection \[\eta: M(\psi_1)\otimes M(\psi_2)\twoheadrightarrow W_A^{\mbox{loc}}(\psi_1 +\psi_2).\]
Then one can show that the image of $M(\psi_1)_{\mu_1}\otimes M(\psi_2)_{\mu_2}$ under the map $\eta$ is zero except for a finite number of weights $\mu_1$ and $\mu_2$ and let $D_i$ be the such finite set of weights. Now for $i =1, 2$, let $M(\psi_i)'$ be the submodule of $M(\psi_i)$ generated by the weight subspaces $M(\psi_i)_{\lambda_i}$ with $\lambda_i \not\in D_i$, and let $\bar{M}(\psi_i) = M(\psi_i)/M(\psi_i)'$. Then $W_A^{\mbox{loc}}(\psi_1 +\psi_2)$ is a quotient of $\bar{M}(\psi_1)\otimes \bar{M}(\psi_2)$. Since $I_{\psi_i}$ has finite co-dimension and there are only a finite number of weights occurring in the quotient $\bar{M}(\psi_i)$, this module is a finite dimensional highest map-weight module of highest map-weight $\psi_i$. Hence, by Theorem \ref{thm4.11}, it is a quotient of $W_A^{\mbox{loc}}(\psi_i)$. Thus, $\bar{M}(\psi_1)\otimes \bar{M}(\psi_2)$ is a quotient of $W_A^{\mbox{loc}}(\psi_1) \otimes W_A^{\mbox{loc}}(\psi_2)$ and this implies $W_A^{\mbox{loc}}(\psi_1 +\psi_2)$ is a quotient of $W_A^{\mbox{loc}}(\psi_1) \otimes W_A^{\mbox{loc}}(\psi_2)$. Since the modules $W_A^{\mbox{loc}}(\psi_1 +\psi_2)$ and $W_A^{\mbox{loc}}(\psi_1) \otimes W_A^{\mbox{loc}}(\psi_2)$ are both finite dimensional, and one is quotient of other implies that $W_A^{\mbox{loc}}(\psi_1 +\psi_2)\cong W_A^{\mbox{loc}}(\psi_1) \otimes W_A^{\mbox{loc}}(\psi_2)$. 
\end{proof}

Note that $W_A^{\mbox{loc}}(\psi_1)$ and $W_A^{\mbox{loc}}(\psi_2)$ satisfy the hypothesis of Theorem \ref{tensorprod}, then 
\[W_A^{\mbox{loc}}(\psi_1)\widehat{\otimes} W_A^{\mbox{loc}}(\psi_2) \cong W_A^{\mbox{loc}}(\psi_1+\psi_2).\]


\end{document}